\documentclass[11pt]{article}
\usepackage[utf8]{inputenc}

\usepackage{amsmath,amssymb,amsthm} 
\usepackage{xcolor}
\usepackage[numbers,sort]{natbib}
\global\long\def\argmin{\operatorname*{argmin}}%

\usepackage[T1]{fontenc}
\usepackage{geometry}
\usepackage{amsmath}
\usepackage{amsthm}
\usepackage{amssymb}
\usepackage{array}
\usepackage{float}
\usepackage{booktabs}
\usepackage{tablefootnote}
\usepackage[title]{appendix}
\usepackage[colorlinks=true, linkcolor=blue, citecolor=blue, urlcolor=blue, anchorcolor=blue]{hyperref}
\allowdisplaybreaks

\definecolor{green}{rgb}{0,0.6,0.0}

\usepackage{amsmath}
\usepackage{amsfonts}
\usepackage{latexsym}
\usepackage{amssymb}

\newtheorem{thm}{Theorem}[section]
\newtheorem{theorem}[thm]{Theorem}

\newtheorem{lemma}[thm]{Lemma}

\newtheorem{proposition}[thm]{Proposition}

\newtheorem{remark}[thm]{Remark}

\newcommand{\beq}{\begin{equation}}
\newcommand{\eeq}{\end{equation}}
\newcommand{\beqa}{\begin{eqnarray}}
\newcommand{\eeqa}{\end{eqnarray}}
\newcommand{\beqas}{\begin{eqnarray*}}
\newcommand{\eeqas}{\end{eqnarray*}}
\newcommand{\ei}{\end{itemize}}

\newcommand{\vgap}{\vspace{.1in}}

\newcommand{\R}{\mathbb{R}}

\newcommand{\lam}{{\lambda}}

\newcommand{\inner}[2]{\langle #1,#2\rangle}

\newcommand{\inte}{\mathrm{int}\,}

\newcommand{\dom}{\mathrm{dom}\,}

\newcommand{\bConv}[1]{\mbox{\rm C}\overline{\mbox{\rm onv}}\,(\Re^{#1})}

\newcommand{\tx}{\tilde x}

\global\long\def\tx{\tilde{x}}%
\global\long\def\rn{\Re^{n}}%
\global\long\def\cal{\mathcal}%
\global\long\def\R{\Re}%
\global\long\def\r{\Re}%
\global\long\def\n{\mathbb{N}}%
\global\long\def\pt{\mathbb{\partial}}%
\global\long\def\lam{\lambda}%
\global\long\def\argmin{\operatorname*{argmin}}%
\global\long\def\dom{\operatorname*{dom}}%
\global\long\def\inner#1#2{\langle#1,#2\rangle}%
\global\long\def\trc{\operatorname*{tr}}%
\global\long\def\cConv{\overline{{\rm Conv}}\ }%
\global\long\def\bConv{\overline{\text{Conv}}(\rn)}%
\usepackage{titlesec}
\titleformat{\section}
  {\normalfont\fontsize{14}{15}\bfseries}{\thesection}{1em}{}
  \titleformat{\subsection}
  {\normalfont\fontsize{12}{15}\bfseries}{\thesubsection}{1em}{}

\title{An adaptive superfast inexact proximal \\
augmented Lagrangian method for smooth \\
nonconvex composite optimization problems}
\author{Arnesh Sujanani\thanks{School of Industrial and Systems Engineering, Georgia Institute of Technology, Atlanta, GA, 30332-0205. (Email: {\tt asujanani6@gatech.edu} \& {\tt monteiro@isye.gatech.edu}). These authors were partially supported by AFORS Grant FA9550-22-1-0088.} \and 
  Renato D.C. Monteiro\footnotemark[1]}

\date{July 24, 2022\\
1st revision: October 3, 2022
}

\begin{document}
\maketitle
\begin{abstract}
This work presents an adaptive superfast proximal augmented Lagrangian (AS-PAL) method for solving linearly-constrained smooth nonconvex composite optimization problems. Each iteration of AS-PAL inexactly solves a possibly nonconvex proximal augmented Lagrangian
(AL) subproblem obtained by
an aggressive/adaptive choice of prox stepsize
with the aim of substantially improving its computational performance followed by a full Lagrangian multiplier update.
A major advantage of AS-PAL compared to other AL methods is that
it requires no knowledge
of parameters (e.g., size of constraint matrix, objective function curvatures, etc) associated
with the optimization problem,
due to its adaptive nature
not only in choosing the prox stepsize
but also in using a crucial
adaptive accelerated composite gradient variant to solve the proximal AL subproblems.
The speed and efficiency of AS-PAL is demonstrated through extensive computational experiments showing that it can solve many instances more than ten times faster than
other state-of-the-art
penalty and AL methods,
particularly when high accuracy
is required.
\end{abstract}
\section{Introduction} \label{sec:intro}
The main goal of this paper 
is to present the theoretical analysis and the excellent computational performance of an
adaptive superfast proximal augmented Lagrangian method,
referred to as AS-PAL,
for solving
the linearly-constrained smooth nonconvex composite optimization (SNCO) problem
	\begin{equation} \label{eq:main_prb}
	\phi^*:=\min \{ \phi(z) :=  f(z) + h(z) : A z  =b\},
	\end{equation}
	where $A:\Re^n \to \Re^l$ is a linear operator, $b\in\Re^l$,
	 $h:\Re^n \to (-\infty,\infty]$ is a closed proper convex function which
	 is $M_h$-Lipschitz continuous on its compact domain, and
	$f:\Re^n \to \Re$ is a real-valued differentiable
	nonconvex function which is
	$m_f$-weakly convex
	and whose
	gradient is $L_f$--Lipschitz continuous.
	AS-PAL is essentially an adaptive
	version of the IAIPAL method and the NL-IAIPAL method
	studied in \cite{NL-IPAL, RenWilmelo2020iteration},
	but, in contrast to these methods, it does not require knowledge of the above parameters $m_f$, $L_f$, and $M_h$.

An iteration of AS-PAL
has a similar pattern to
the ones of the methods in \cite{RenWilmelo2020iteration, NL-IPAL} and is
also based on the augmented Lagrangian (AL) function ${\cal L}_c(z;p)$ defined as
\begin{equation}\label{lagrangian2}
{\cal L}_c(z;p):=f(z)+h(z)+\left\langle p,Az-b\right\rangle+\frac{c}{2}\|Az-b\|^2.
\end{equation}
More specifically, its rough description is as follows:
given $(z_{k-1},p_{k-1}) \in {\cal H} \times \Re^l$ and
a pair of positive scalars $(\lam_k,c_k)$, it computes $z_k$ as a suitable approximate solution of the possibly nonconvex proximal subproblem
\begin{equation}\label{eq:approx_primal_update}
\min_{u}\left\{ \lam_k {\cal L}_{c_k}(u;p_{k-1})+\frac{1}{2}\|u-z_{k-1}\|^{2}\right\},
\end{equation}
and $p_k$ according to the full Lagrange multiplier update
\begin{equation}\label{eq:dual_update}
p_{k}=p_{k-1}+c_{k}(Az_{k}-b).
\end{equation}
Based on the fact that \eqref{eq:approx_primal_update} is strongly convex whenever the prox stepsize $\lambda_k$
is chosen in $(0,1/m_f)$, the methods of \cite{RenWilmelo2020iteration, NL-IPAL} set $\lambda_k=0.5/m_f$ for every $k$ and solve each strongly-convex subproblem using an accelerated composite gradient (ACG) method (see \cite{fistaReport2021, MontSvaiter_fista}).

{\bf Our contributions:}
Since it is empirically observed
that the larger $\lam_k$ is, the faster the procedure outlined above in \eqref{eq:approx_primal_update}-\eqref{eq:dual_update} approaches
a desired approximate solution of \eqref{eq:main_prb},
AS-PAL adaptively chooses the prox stepsize $\lam_k$ to be a scalar
which 
is usually much larger than $0.5/m_f$.
As \eqref{eq:approx_primal_update} may become nonconvex with such a choice of $\lam_k$, a standard ACG method
applied to \eqref{eq:approx_primal_update}
may fail to obtain a desirable
approximate solution of \eqref{eq:approx_primal_update}.
To remedy this situation, AS-PAL uses a new adaptive ACG method for solving \eqref{eq:approx_primal_update} which accounts for the fact that \eqref{eq:approx_primal_update}
may be nonconvex and the Lipschitz constant of the objective function of \eqref{eq:approx_primal_update} may be unknown.
Thus, in contrast to the methods of \cite{RenWilmelo2020iteration, NL-IPAL},
AS-PAL has the interesting feature
of requiring no knowledge of the parameters $m_f$, $L_f$ and $M_h$ underlying \eqref{eq:approx_primal_update}
in view of its ability to
adaptively generate the prox stepsize $\lam_k$ and
the
estimate of the Lipschitz constant of the objective function of \eqref{eq:approx_primal_update} within the adaptive ACG method.
Moreover, as was shown for
the method of \cite{NL-IPAL}, under the assumption that a Slater point exists,
it is also shown 
that, 
for any given tolerance pair $(\hat \rho,\hat \eta) \in \Re^2_{++}$, AS-PAL finds a $(\hat \rho, \hat \eta)$-approximate stationary solution of \eqref{eq:approx_primal_update}, i.e., a triple $(z,p,w)$  
satisfying
\begin{align}\label{eq:approx_stationary1'}
	w\in\nabla f(z)+\partial h(z)+A^*p ,\quad
\|w\|\leq \hat \rho,\quad  \|Az-b\|\leq \hat \eta,
\end{align}
in at most $\cal O(\hat \eta^{-1/2} \hat \rho^{-2}+\hat \rho^{-3})$ iterations (up to logarithmic terms). Finally, a major advantage of AS-PAL is that it substantially improves the computational performance of the methods in \cite{NL-IPAL, RenWilmelo2020iteration}, whose performance was already substantially better than other existing methods for solving \eqref{eq:main_prb}. Our extensive computational results of section~\ref{sec:numerical} show that AS-PAL can efficiently compute highly accurate solutions for all problems tested, while the other methods can fail to do so in many of these problems. AS-PAL can often find such solutions in just a few seconds or minutes while all the other methods may take several hours to do so.

\par {\bf Literature review.}
We only focus on relatively recent papers dealing with the iteration complexity of augmented Lagrangian (AL) type methods.
In the convex setting, AL-based methods have been widely studied for example in \cite{LanRen2013PenMet, Aybatpenalty, IterComplConicprog, AybatAugLag, LanMonteiroAugLag, ShiqiaMaAugLag16, zhaosongAugLag18, Patrascu2017, YangyangAugLag17}. 


We now discuss AL type methods in the nonconvex setting of \eqref{eq:main_prb}.
Various 
proximal AL methods for solving both linearly and
nonlinearly constrained SNCO problems
have been studied in \cite{HongPertAugLag, AIDAL, RJWIPAAL2020, NL-IPAL, RenWilmelo2020iteration, YinMoreau, ErrorBoundJzhang-ZQLuo2020, ADMMJzhang-ZQLuo2020, DualDescent}.
More specifically, \cite{AIDAL, HongPertAugLag,RJWIPAAL2020} present proximal AL methods based on a perturbed augmented Lagrangian function and an under-relaxed multiplier update. Papers \cite{RenWilmelo2020iteration, NL-IPAL} present an accelerated proximal AL method based off the classical augmented Lagrangian function and a full multiplier update. The method in \cite{DualDescent} is an AL-based method which reverses the direction of the multiplier update. Papers \cite{YinMoreau, ErrorBoundJzhang-ZQLuo2020, ADMMJzhang-ZQLuo2020} study AL type variants based on the Moreau envelope. Finally,
non-proximal AL methods for solving SNCO problems
are studied in \cite{ImprovedShrinkingALM20, inexactAugLag19}.

We now discuss papers that are tangentially related to this work. Penalty methods for SNCO problems have been studied in \cite{WJRproxmet1, WJRComputQPAIPP, MinMax-RenWilliam, Penalty2.5Xu}. It is worth mentioning that AS-PAL extends the methods in \cite{RenWilmelo2020iteration, NL-IPAL} by allowing for an adaptive prox stepsize, similar to the way the method of \cite{WJRComputQPAIPP} extends the one in \cite{WJRproxmet1}. Finally, paper \cite{SZhang-Pen-admm} studies a penalty-ADMM method that solves an equivalent reformulation of \eqref{eq:main_prb} 
while the paper \cite{HybridPenaltyAugLag19} presents an inexact proximal point method applied to the function defined as $\phi(z)$ if $z$ is feasible and $+\infty$ otherwise.

\par {\bf Organization of the paper.} The paper
is laid out as follows. Subsection~\ref{subsec:notation} presents basic definitions and notation used throughout the paper. Section~\ref{sec:aug_lagr} contains two subsections. The first describes the problem of interest and the assumptions made on it. The second formally states the AS-PAL method and its main complexity result. Section~\ref{sec:proofs of main proposition} is dedicated to proving the main complexity result. Section~\ref{sec:numerical} presents extensive computational experiments which demonstrate the efficiency of AS-PAL. The Appendix contains two subsections. Appendix~\ref{sec:acg} presents the ADAP-FISTA algorithm which is used to solve possibly nonconvex unconstrained subproblems while Appendix~\ref{technical lagrange multiplier} presents technical results which are used to prove that the sequence of the Lagrange multipliers generated by AS-PAL is bounded.

\subsection{Basic Definitions and Notations} \label{subsec:notation}

This subsection presents notation and basic definitions used in this paper.
	
Let $\Re_{+}$  and $\Re_{++}$ denote the
set of nonnegative and positive
real numbers, respectively. We denote by  $\Re^n$  an $n$-dimensional inner product space with inner product and associated norm denoted by $\inner{\cdot}{\cdot}$ and $\|\cdot\|$, respectively. We use $\Re^{l\times n}$ to denote the set of all $l\times n$ matrices and ${\mathbb S}_n^+$ to denote the set of positive semidefinite matrices in $\r^{n\times n}$.  The smallest positive singular value of a  nonzero linear operator $Q:\Re^n\to \Re^l$ is denoted by $\nu^+_Q$. For a given closed convex set $Z \subset \Re^n$, its boundary is denoted by $\partial Z$ and the distance of a point $z \in \Re^n$  to $Z$ is denoted by ${\rm dist}(z,Z)$. The indicator function of $Z$, denoted by $\delta_Z$, is defined by $\delta_Z(z)=0$ if $z\in Z$, and $\delta_Z(z)=+\infty$ otherwise.
 For any $t>0$ and $b\geq 0$, we let $\log_b^+(t):=\max\{\log t, b\}$, and we  define ${\cal O}_1(\cdot) = {\cal O}(1 + \cdot)$.

The domain of a function $h :\Re^n\to (-\infty,\infty]$ is the set $\dom h := \{x\in \Re^n : h(x) < +\infty\}$.
Moreover, $h$ is said to be proper if
$\dom h \ne \emptyset$. The set of all lower semi-continuous proper convex functions defined in $\Re^n$ is denoted by $\cConv \rn$. The $\varepsilon$-subdifferential of a proper function $h :\Re^n\to (-\infty,\infty]$ is defined by 
\begin{equation}\label{def:epsSubdiff}
\partial_\varepsilon h(z):=\{u\in \Re^n: h(z')\geq h(z)+\inner{u}{z'-z}-\varepsilon, \quad \forall z' \in \Re^n\}
\end{equation}
for every $z\in \Re^n$.	The classical subdifferential, denoted by $\partial h(\cdot)$,  corresponds to $\partial_0 h(\cdot)$.  
Recall that, for a given $\varepsilon\geq 0$, the $\varepsilon$-normal cone of a closed convex set $C$ at $z\in C$, denoted by  $N^{\varepsilon}_C(z)$, is 
$$N^{\varepsilon}_C(z):=\{\xi \in \Re^n: \inner{\xi}{u-z}\leq \varepsilon,\quad \forall u\in C\}.$$
The normal cone of a closed convex set $C$ at $z\in C$ is denoted by  $N_C(z)=N^0_C(z)$.
If $\psi$ is a real-valued function which
is differentiable at $\bar z \in \Re^n$, then its affine   approximation $\ell_\psi(\cdot,\bar z)$ at $\bar z$ is given by
\begin{equation}\label{eq:defell}
\ell_\psi(z;\bar z) :=  \psi(\bar z) + \inner{\nabla \psi(\bar z)}{z-\bar z} \quad \forall  z \in \Re^n.
\end{equation}

\section{The AS-PAL method} \label{sec:aug_lagr}

This section consists of two subsections. The first one precisely describes the problem of interest and its assumptions. The second one motivates and states the AS-PAL method and presents its main complexity result.

\subsection{Problem of Interest} \label{subsec:prb_of_interest}

This subsection presents the main problem of interest and discusses the assumptions underlying it.

Consider problem \eqref{eq:main_prb}
where $A: \Re^{n} \rightarrow \Re^{l}$, $b \in \Re^{l}$ and
functions $f,h: \Re^{n} \rightarrow (-\infty,\infty]$ satisfy the following
assumptions:
\begin{itemize}
\item[{\bf(A1)}] $h\in \bConv$ is $M_h$-Lipschitz continuous on ${\mathcal H}:=\dom h$
and the diameter
\[
D_h:=\sup \{ \|z-z'\| : z, z' \in {\mathcal H} \}
\]
of ${\mathcal H}$ is finite;
\item[{\bf(A2)}] $A$ is a nonzero linear operator and there exists $\bar z \in \inte({\mathcal H})$ such that $A\bar z=b$;
\item[{\bf(A3)}] $f$ is nonconvex and  differentiable on $\Re^{n}$, and there exists $L_f\geq m_f > 0$ such that for all $z,z' \in \R^n$,
\begin{align}
\|\nabla f(z') -  \nabla f(z)\| &\le L_f \|z'-z\|\label{gradLips},\\
f(z') -\ell_f(z';z) &\ge - \frac{m_f}2 \|z'-z \|^2.\label{lowerCurvature-m}
\end{align}
\end{itemize} 




\subsection{The AS-PAL method}\label{sec:ASPAL}

This subsection  motivates and states the AS-PAL method and presents its main complexity result.  

Before presenting the method, we give a short but precise outline of its key steps as well as a description of how its iterates are generated.
Recall from the introduction that the AS-PAL method,
whose goal is to find a $(\hat \rho, \hat \eta)$-approximate stationary solution as in \eqref{eq:approx_stationary1'}, is an iterative method which, at its $k$-th step, computes its next iterate $(z_k, p_k)$ according to \eqref{eq:approx_primal_update} and \eqref{eq:dual_update}. We are now ready to provide a complete description of the AS-PAL method.

\noindent\begin{minipage}[t]{1\columnwidth}%
\rule[0.5ex]{1\columnwidth}{1pt}

\noindent \textbf{AS-PAL Method}

\noindent \rule[0.5ex]{1\columnwidth}{1pt}%
\end{minipage}

\noindent \textbf{Input}: functions $(f,h)$, scalars
$\sigma\in (0,1/2)$, $\chi \in (0,1)$, and $\beta>1$, an initial $\bar{\lambda}>0$, an initial
point $z_0 \in {\cal H}$, $p_0=0$, a penalty parameter $c_1>0$, and a tolerance pair $(\hat{\rho},\hat{\eta})\in\r_{++}^{2}$.

\noindent \textbf{Output}: a triple $(z,p,w)$ satisfying \eqref{eq:approx_stationary1'}.

\begin{itemize}
\item[{\bf 0.}] set  $\hat k=1$,  $k=1$, and
\begin{equation}\label{def:lamb-C1}
\lambda=\overline{\lambda}, \quad C_{\sigma}= \frac{2(1-\sigma)^2}{1-2\sigma};
\end{equation}

\item[{\bf 1.}]  let $M^k_0 \in [1,\overline \lambda(L_f+ c_k\|A\|^2)+1]$ and call the ADAP-FISTA method described in Appendix~\ref{sec:acg}
with inputs 
\begin{align}
x_0&=z_{k-1}, \quad
(\mu,L_0,\chi,\beta,\sigma)= (1/2,M_0^k,\chi,\beta,\sigma)\label{eq:Ms-mu}, \\
\psi_s &=\lam [\mathcal{L}_{c_k}(\cdot,p_{k-1})-h] +\frac{1}{2}\|\cdot-z_{k-1}\|^2 , \quad \psi_n = \lam h \label{eq:psiS-psimu};
\end{align}
\item[{\bf 2.}]
if ADAP-FISTA fails or its output 
$(z,u)$ (if it succeeds) does not satisfy the inequality \begin{equation}\label{subdiff ineq check}
	   \lambda{\cal L}_{c_k}(z_{k-1},p_{k-1}) - \left [ \lambda{\cal L}_{c_k}(z,p_{k-1}) + \frac{1}{2} \|z-z_{k-1}\|^2 \right ] \ge \inner{u}{z_{k-1}-z},
	   \end{equation}
then set $\lam=\lam/2$ and go to step $1$; else, set
	$(z_{k},u_{k})=(z,u)$,
	$\lam_k=\lam$,
	and 
	\begin{align}
	{w}_{k}:&=\frac{u_k+z_{k-1}-z_{k}}{\lambda_k},\label{eq:w_k_def}\\
	p_{k}:&=p_{k-1}+c_{k}(Az_k-b)\label{eq:dual_update2},
	\end{align}
	and
	go to step~3;
\item[{\bf 3.}]	
if $\|w_k\|\leq \hat \rho$ and $\|Az_k-b\|\leq \hat \eta$,
then stop with success and output $(z,p,w)=(z_k,p_k,w_k)$; else, go to step~4;

\item[{\bf 4.}]
if $k\geq \hat k+1 $ and 
\begin{equation}\label{eq:stopCOndition-deltak}
\Delta_k:=\frac{1}{\sum_{i=\hat k+1}^{k}\lambda_i}\left[\mathcal{L}_{c_k}(z_{\hat k},p_{\hat k-1})-\mathcal{L}_{c_k}(z_k,p_k) - \frac{\|p_k\|^2}{2c_k}\right] \le
\max\left\{\frac{\sum_{i=\hat k+1}^{k}\lam_{i}\|w_{i}\|^2}{2C_{\sigma}\sum_{i=\hat k+1}^{k}\lambda_i},\frac{\hat \rho^2}{2C_\sigma}\right\},
    \end{equation}
then set $c_{k+1}=2 c_k$ and $\hat k=k+1$;
otherwise, set $c_{k+1}=c_k$;
\item[{\bf 5.}] 
set $k\gets k+1$ and go to step~1. 
\end{itemize}
\noindent \rule[0.5ex]{1\columnwidth}{1pt}

AS-PAL makes two types of iterations, namely, the outer iterations indexed by $k$ and the ACG iterations performed during its calls to the ADAP-FISTA method in step 1.

We now make some remarks about AS-PAL. First, it follows from Proposition~\ref{prop:nest_complex1} (see Appendix~\ref{sec:acg}) that the total number of resolvent evaluations \footnote{A resolvent evaluation of $h$ is an evaluation of $(I+\gamma \partial h)^{-1}(\cdot)$ for some $\gamma>0$.} made by ADAP-FISTA is on the same order of magnitude as its total number of ACG iterations. Second, noting that
the sum of
the functions $\psi_s$ and $\psi_n$ 
in \eqref{eq:psiS-psimu} is equal
to the objective function of \eqref{eq:approx_primal_update},
it follows from
Proposition~\ref{prop:nest_complex1} in Appendix~\ref{sec:acg} that the pair $(z_k,u_k)$ in step 2 of AS-PAL is
an approximate solution of \eqref{eq:approx_primal_update} in the sense of \eqref{acg problem}.
 Third, it will be shown in Proposition~\ref{ACG facts}(b) below that the triple $(z_k,p_k,w_k)$ computed in step 2 satisfies the inclusion in \eqref{eq:approx_stationary1'} for every $k \ge 1$. As a consequence, if AS-PAL terminates in step 3, then the triple $(z,p,w)$ output in this step is a $(\hat \rho, \hat \eta)$-approximate solution of \eqref{eq:main_prb}.
 Finally, step 4 provides a test, namely, inequality \eqref{eq:stopCOndition-deltak}, to determine when to increase the penalty parameter $c_k$.

Define
the $l$-th cycle ${\cal C}_l$
as the $l$-th set of consecutive indices
$k \ge 1$ for which $c_k$ remains constant, i.e.,
\begin{equation}\label{def:ctilde-l}
{\cal C}_l := \{ k \ge 1: c_k=\tilde c_l:= 2^{l-1} c_1 \} \quad \forall l \ge 1.
\end{equation}
For every $l \ge 1$,  let
$k_l$ denote the smallest  index in ${\cal C}_l$.
Hence,
\begin{equation}\label{def C}
{\cal C}_l = \{ k_{l}, \ldots, k_{l+1}-1\} \quad \forall l \ge 1.
\end{equation}
Clearly, the different values of $\hat k$ that
arise in step 4 are exactly the indices in $\{k_l : l \ge 1\}$.
Moreover, in view of the test performed
in step 4, we have that
$k_{l+1}-k_{l} \ge 2$ for every $l \ge 1$, or equivalently,
every cycle contains at least two indices. While generating the indices
in the $l$-th cycle,
if an index $k \ge k_{l}+2$
satisfying
\eqref{eq:stopCOndition-deltak} is found,
$k$ becomes the last index $k_{l+1}-1$ in the
$l$-th cycle and the $(l+1)$-th cycle
is started at iteration $k_{l+1}$ with
the penalty parameter set to
$\tilde c_{l+1}=2 \tilde c_l$, where $\tilde c_l$ is as in \eqref{def:ctilde-l}.

In the remaining part of this section, we state the main complexity result for AS-PAL, whose proof is the main focus of Section~\ref{sec:proofs of main proposition}. Before stating the main result, we first introduce
the following quantities:
\begin{equation}\label{phi* and lambda bound and bar d}
\phi_*:=\inf_{z\in \Re^n} \phi(z), \quad \bar{d}:=\mbox{\rm dist}(\bar z, {\partial {\mathcal H}}), \quad \underline \lambda:=\min\{\overline{\lambda}, 1/(4m_f)\}
\end{equation}
\begin{equation}\label{nablaf and kappap and chat}
\nabla_f:=\sup_{z\in \mathcal H} |\nabla f(z)|, \quad \kappa_p:=\frac{2D_h(M_h+\nabla_f+\underline \lambda^{-1}(1+\sigma)D_h)}{\bar d \nu^{+}_A}, \quad \hat{c}(\hat \rho, \hat \eta):= \frac{18C_{\sigma}\kappa^2_p}{\underline \lambda \hat \rho^2}+\frac{2\kappa_p}{\hat \eta}
\end{equation}
\begin{equation}\label{S and kappad}
S:=\sup_{z\in \mathcal H}|\phi(z)|, \quad \kappa_d:=S+\frac{4\kappa^2_p}{c_1}-\phi_*,
\end{equation}
where $\overline \lambda$, $c_1$, and $\sigma$ are input parameters for AS-PAL, $(m_f, L_f)$ are as in (A3), $\bar z$ is as in (A2), $M_h$ is as in (A1), $D_h$ is as in (A1), $\nu^+_A$ is as in Subsection~\ref{subsec:notation}, and $C_\sigma$ is as in \eqref{def:lamb-C1}. Note that
assumptions (A1) and (A3) imply that
$S$ and $\nabla_f$ are finite.

The following result describes the ACG iteration/resolvent evaluation complexity for AS-PAL.

\begin{theorem}\label{theor:StaticIPAAL}
Let a tolerance pair $(\hat \rho, \hat \eta) \in \Re_{++}^{2}$ be given
and assume that
$c_1 \leq 4\hat c(\hat \rho, \hat \eta)$ and
$\overline \lam$ is such that
$\overline{\lam} = \Omega(m_f^{-1})$ and $\log^{+}_{0}(m_f\overline \lambda)\leq {\cal O}(1+\kappa_d/(\underline \lambda \hat \rho^2))$, where $c_1$ and $\overline \lambda$ are the initial penalty parameter and prox stepsize of AS-PAL, respectively, $m_f$ is as in (A3), $\hat c(\hat \rho, \hat \eta)$ is as in \eqref{nablaf and kappap and chat}, and $\kappa_d$ is as in \eqref{S and kappad}. Then,  AS-PAL outputs a $(\hat \rho,\hat\eta)$-approximate stationary solution of \eqref{eq:main_prb} in
\begin{equation}\label{eq:main complexity-bound}
{\cal O}\left(\left[1+\frac{m_f \kappa_d}{\hat \rho^2}\right] \sqrt{\mathcal M(\hat c)} \left[\log \left(\mathcal M(\hat c)+\frac{\hat c}{c_1}\right) \right]^2\right)
\end{equation}
ACG iterations/resolvent evaluations, where $\hat c := \hat c(\hat \rho,\hat \eta)$ and
\begin{equation}\label{M equation}
\mathcal M(c):=\overline{\lambda} (L_f+c\|A\|^2) +1
\quad \forall c \in \R.
\end{equation}
\end{theorem}
It follows from the definitions of $\hat c(\cdot, \cdot)$ and $\mathcal M(\cdot)$ in \eqref{nablaf and kappap and chat} and \eqref{M equation}, respectively, that the iteration complexity bound \eqref{eq:main complexity-bound} in terms of the tolerance pair $(\hat\rho, \hat \eta)$, up to a logarithmic term, is
\[
{\cal O}\left(\frac{1}{\sqrt{\hat \eta}\cdot\hat\rho^2}+\frac{1}{\hat\rho^3}
\right).
\]

\section{Proof of Theorem~\ref{theor:StaticIPAAL}}\label{sec:proofs of main proposition}

The result below describes properties of the loop consisting of steps 1 and 2 of AS-PAL.

\begin{proposition}\label{ACG facts}
Let $k \in {\cal C}_l$
for some $l \ge 1$ be given.
Then, the following statements hold:
\begin{itemize}
\item[(a)] every ACG call in step 1 of the $k$-th iteration of AS-PAL performs
\begin{equation}\label{ACG result-2}
{\cal O}_1 \left( \sqrt{\mathcal M(\tilde c_l)}\, \log \mathcal M(\tilde c_{l})
\right)
\end{equation}
ACG iterations/resolvent evaluations;
\item [(b)]
during the $k$-th iteration of AS-PAL,
the loop consisting of steps 1 and 2 eventually ends with
a quintuple $(z_k,u_k,w_k,p_k,\lam_k)$ satisfying
\begin{align}
&\|u_k\| \le \sigma \min \left\{\|z_{k}-z_{k-1}\| \,,\,\frac{\|\lambda_k w_k\|}{1-\sigma}\right\};\label{Nesterov zk}\\
&\lambda_k{\cal L}_{c_k}(z_{k-1},p_{k-1}) - \left [ \lambda_k{\cal L}_{c_k}(z_{k},p_{k-1}) + \frac{1}{2} \|z_k-z_{k-1}\|^2 \right ] \ge \inner{u_{k}}{z_{k-1}-z_k};\label{subdiff zk} \\
&w_k\in\nabla f(z_k)+\pt h(z_k)+A^{*}p_k, \quad\|\lambda_k w_k\|\leq (1+\sigma)\|z_k-z_{k-1}\|; \label{w inclusion}\\
&\overline{\lambda}\geq \lambda_k \geq \underline \lambda,\label{lower lambda bound}
\end{align}
where $\overline \lambda$ is the initial prox stepsize and $\underline \lambda$ is as in \eqref{phi* and lambda bound and bar d};
moreover, every prox stepsize $\lam$ generated in the loop consisting of steps 1 and 2 of AS-PAL is in $[\underline \lam, \overline \lam]$.
\end{itemize}
\end{proposition}

\begin{proof}
(a) Using the definition of $\cal L_{c_k}(\cdot;p_{k-1})$
in \eqref{lagrangian2}
and assumption \eqref{gradLips}, we easily see that its smooth part, namely, $\cal L_{ c_k}(\cdot;p_{k-1})-h(\cdot)$,
has $(L_f+c_k\|A\|^2)$-Lipschitz continuous gradient everywhere on $\R^n$.
This observation together with the facts that $\lam \le \overline \lam$,
$c_k = \tilde c_l$,
and the definition of
$\cal M(\cdot)$ in \eqref{M equation}, then imply that the function
$\psi_s$ in \eqref{eq:psiS-psimu} has
$\cal M(\tilde c_l)$-Lipschitz continuous gradient.
Noting that $M_0^k$ in step 1
is chosen so that $M_0^k \le \mathcal M(c_k) =\mathcal M(\tilde c_l)$ and
that each call to ADAP-FISTA in step 1 is made
with $(\mu,L_0)=(1/2,M_0^k)$, we then conclude that (a) follows directly from Proposition~\ref{prop:nest_complex1}(a) with $\bar L= \mathcal M(\tilde c_l)$ and $\mu=1/2$.

(b) 
We first claim that if the loop consisting of steps 1 and 2 of the $k$-iteration of AS-PAL
stops, then \eqref{Nesterov zk}, \eqref{subdiff zk}, and \eqref{w inclusion} hold.
Indeed, assume that the loop consisting of steps 1 and 2 of the $k$-th iteration of AS-PAL
stops. It then follows
that ADAP-FISTA with inputs given by
\eqref{eq:Ms-mu} and \eqref{eq:psiS-psimu} stops successfully
and $(z,u,\lam)=(z_k,u_k,\lam_k)$ satisfies \eqref{subdiff ineq check}.
These two conclusions,
identities \eqref{eq:psiS-psimu} and \eqref{eq:w_k_def}, and Proposition~\ref{prop:nest_complex1}(b) with $(\psi_s,\psi_n)$ as in \eqref{eq:psiS-psimu}, $x_0=z_{k-1}$, and $(y,u)=(z_k,u_k)$
then imply that
\eqref{subdiff zk},
the first inequality in \eqref{Nesterov zk},
and the inclusion in \eqref{w inclusion} hold.
Now, using the definition of $w_k$ in \eqref{eq:w_k_def}, the triangle inequality, and the first inequality in \eqref{Nesterov zk}, we have:
\begin{equation}\label{lambdaw bound}
\frac{1}{\sigma}\|u_k\|-\|u_k\|\overset{\eqref{Nesterov zk}}{\leq}\|z_{k}-z_{k-1}\|-\|u_{k}\|\leq \|u_k+z_{k-1} -z_{k}\|\overset{\eqref{eq:w_k_def}}{=}\|\lambda_kw_k\|\overset{\eqref{Nesterov zk}}{\leq} (1+\sigma)\|z_k-z_{k-1}\|,
\end{equation}
from which 
the second inequality in \eqref{Nesterov zk} and the inequality in \eqref{w inclusion} follow.

We now claim that
if 
step 1 is performed
with a prox stepsize
$\lambda \leq 1/(2m_f)$
in the $k$-th iteration,
then for every $j >k$, we have that
$\lam_{j-1}=\lam$
and the $j$-th iteration performs step 1 only once.
To show the claim,
assume that $\lambda\leq 1/(2m_f)$. Using this assumption, the definition of $\cal L_c$ in \eqref{lagrangian2}, and the assumption \eqref{lowerCurvature-m} that $f$ is $m_f$-weakly convex, we see that the function $\psi_s$ in \eqref{eq:psiS-psimu} is strongly convex with modulus $1-\lam m_f \geq 1/2$. Since each ACG call is performed in step~1 of AS-PAL with $\mu=1/2$, it follows immediately from Proposition~\ref{strongly convex fista} with $(\psi_s,\psi_n)$ as in \eqref{eq:psiS-psimu} that ADAP-FISTA terminates successfully and outputs
a pair $(z,u)$ satisfying $u \in \partial (\psi_s+\psi_n)(z)$. This inclusion, the definition of
$(\psi_s,\psi_n)$, and 
the definition of subdifferential in \eqref{def:epsSubdiff},
then imply that
\eqref{subdiff ineq check} holds.
Hence, in view of the termination criteria of step~2 of AS-PAL, it follows that $\lambda_k=\lambda$. It is then easy to see, by the way $\lambda$ is updated in step~2 of AS-PAL, that $\lambda$ is not halved in the $(k+1)$-th iteration or any subsequent iteration, hence proving the claim.

It is now straightforward to
see that the above two claims,
the fact that the initial value of the prox stepsize is equal to $\overline{\lambda}$, and
the way $\lam_k$ is updated in
AS-PAL, imply
that the lemma holds.
\end{proof}

The subsequent technical result characterizes the change in the augmented Lagrangian function between consecutive iterations of the AS-PAL method.



\begin{lemma}\label{lem:declemma5}
For every $k \ge 1$, we have:
\begin{align}
{\cal L}_{c_k}(z_k,p_k) - {\cal L}_{c_k}(z_k,p_{k-1}) &=\frac{1}{c_k}\|p_k-p_{k-1}\|^2, \label{auxineq:declemma2}
\end{align}
and
\begin{align}
\frac {\lambda_k}{C_{\sigma}}\|w_k\|^2
&\le {\cal L}_{c_k}(z_{k-1},p_{k-1})-{\cal L}_{c_k}(z_k,p_k)+\frac{1}{c_k}\|p_k-p_{k-1}\|^2\label{auxineq:declemma3}
\end{align}
where $C_\sigma$ is as in \eqref{def:lamb-C1}.
 \end{lemma}
\begin{proof}
Identity \eqref{auxineq:declemma2} follows immediately from the definition of the Lagrangian in \eqref{lagrangian2} and relation \eqref{eq:dual_update2}. Now, using relation \eqref{subdiff zk},
the second inequality in \eqref{Nesterov zk}, and
the definitions of $C_{\sigma}$ and $w_k$ in \eqref{def:lamb-C1} and \eqref{eq:w_k_def}, respectively, we conclude that:
\begin{align}
& \lambda_k{\cal L}_{c_k}(z_{k-1},p_{k-1})-\lambda_k{\cal L}_{c_k}(z_k,p_{k-1})  \overset{\eqref{subdiff zk}}{\geq} \frac{1}{2} \|z_k-z_{k-1}\|^2 +\inner{u_{k}}{z_{k-1}-z_k} \nonumber\\
&=\frac12 \| z_{k-1}-z_k + u_k \|^2-\frac{1}{2}\|u_k\|^2 \overset{\eqref{eq:w_k_def}}{=} \frac{1}{2}\|\lambda_kw_k\|^2-\frac{1}{2}\|u_k\|^2 \nonumber \\
&\overset{\eqref{Nesterov zk}}{\geq} \frac{1}{2}\|\lambda_kw_k\|^2-\frac{\sigma^2}{2(1-\sigma)^2}\|\lambda_kw_k\|^2=\frac{1-2\sigma}{2(1-\sigma)^2}\|\lambda_kw_k\|^2 \overset{\eqref{def:lamb-C1}}{=} \frac{\|\lambda_kw_k\|^2}{C_{\sigma}}. \label{eq:aux_DeltaLagr_bd1}
\end{align}
Inequality \eqref{auxineq:declemma3} now follows by dividing
\eqref{eq:aux_DeltaLagr_bd1} by
$\lambda_k$ and combining
the resulting inequality with
\eqref{auxineq:declemma2}.
\end{proof}

The result below,
which establishes boundedness of the sequence of Lagrange multipliers,
makes use of a
technical result in the Appendix,
namely Lemma \ref{lem:qbounds-2}.

\begin{proposition}\label{th:pkbounded} The sequence  $\{p_k\}$ generated by  AS-PAL satisfies 
\begin{equation} \label{ineq:pkbounded}
\|p_k\|\leq \kappa_p,\quad \forall k\geq 0,
\end{equation}
where $\kappa_p$ is defined in \eqref{nablaf and kappap and chat}.
\end{proposition}
\begin{proof}
Using the inequality in \eqref{w inclusion}, the triangle inequality, the second inequality in \eqref{lower lambda bound}, and the definitions of $D_h$ and $\nabla_f$ in (A1) and \eqref{nablaf and kappap and chat}, respectively, we conclude that
\begin{equation}\label{new w bound}
\|w_k-\nabla f(z_k)\|\overset{\eqref{w inclusion}}{\leq} \frac{1}{\lam_k}(1+\sigma)\|z_k-z_{k-1}\|+\nabla_f\overset{\eqref{lower lambda bound}}{\leq} \frac{D_h(1+\sigma)}{\underline \lambda}+\nabla_f.
\end{equation}
Now, using the inclusion in \eqref{w inclusion}, the relation in \eqref{new w bound},
Lemma~\ref{lem:qbounds-2}(b) with
$(z,q,r)=(z_k,p_k,w_k-\nabla f(z_k))$ and
$q^-=p_{k-1}$, and the
definition of $\kappa_p$ in \eqref{nablaf and kappap and chat}, we conclude that for every
$k \geq 1$:
\begin{equation}\label{pksecondbound}
\|p_k\|\overset{\eqref{q bound-2}}{\leq} \max\left\{\|p_{k-1}\|,\frac{2D_h(M_h+\|w_k-\nabla f(z_k)\|)}{\bar d \nu^{+}_A} \right\}\overset{\eqref{new w bound}}{\leq}\max\left\{\|p_{k-1}\|,\kappa_p\right\}.
\end{equation}
Now, the conclusion of the proposition follows from the above relation, the fact that
$p_0=0$, and a simple induction argument.
\end{proof}

Recall that the
$l$-th cycle ${\cal C}_l$ of AS-PAL is  defined
in \eqref{def:ctilde-l}. The following result shows that  the sequence $\{\|\|w_k\|\}_{k\in {\cal C}_l}$ is bounded and can be controlled  by   $\{\Delta_k\}_{k\in {\cal C}_l}$ plus a term which is of $\mathcal{O}(1/{\tilde c_l})$.

\begin{lemma}\label{lem:minwk-Deltak}
Consider the sequences
$\{(z_k,p_k,w_k)\}_{k\in {\cal C}_l}$ and $\{\Delta_k\}$ generated by AS-PAL. Then, for every $k\in {\cal C}_l$ such that $k \geq \hat k+1$, we have:
  
\begin{equation}\label{main-ineq-dreasingLag}
\frac{\sum_{i=\hat k+1}^{k} \lambda_i\|w_i\|^2}{\sum_{i=\hat k+1}^{k} \lambda_i}\leq C_{\sigma}\left(\Delta_k + \frac{9 \kappa^2_p}{\underline \lambda \tilde c_l}\right)
\end{equation}
where  $C_{\sigma}$, $\underline \lambda$, and $\kappa_p$ are as in \eqref{def:lamb-C1}, \eqref{phi* and lambda bound and bar d}, and \eqref{nablaf and kappap and chat}, respectively and $\hat k$ is the first index in ${\cal C}_l$.

\end{lemma}
\begin{proof}
We have by relation \eqref{ineq:pkbounded} and the bound   $\|p_j-p_{j-1}\|^2\leq 2\|p_j\|^2+2\|p_{j-1}\|^2\leq 4 \kappa^2_p$, that it follows that for any $k\in \mathcal C_{l}$,
\begin{equation}\label{eq:auxboundpk2}
\frac{\|p_k\|^2}{2}+ \sum_{i={\hat k}}^k\|p_i-p_{i-1}\|^2
\le \frac{\kappa_p^2}{2}+ 4(k-\hat k+1)\kappa_p^2= \frac{(1+ 8(k-\hat k+1))\kappa_p^2}{2}\leq 9(k-\hat k)\kappa_p^2.
\end{equation}
Hence, relations \eqref{auxineq:declemma2}, \eqref{auxineq:declemma3}, and \eqref{eq:auxboundpk2} and the fact that $c_{k}=\tilde c_l$ for every $k\in {\cal C}_l$,  imply that for any $k\in {\cal C}_l$ such that $k \geq \hat k+1$,
\begin{align*}
	\frac{1-2\sigma}{2(1-\sigma)^2}\sum_{i=\hat k+1}^k \lam_i\|w_i\|^2 & \overset{\eqref{auxineq:declemma3}}{\leq} \sum_{i=\hat k+1}^k\left[\
	{\cal L}_{c_i}(z_{i-1},p_{i-1})-{\cal L}_{c_i}(z_i,p_i)+ \frac{1}{c_i}\|p_i-p_{i-1}\|^2 \right]
	\\
    & \overset{j \in {\cal C}_l}{=} \sum_{i=\hat k+1}^k
	\left[{\cal L}_{\tilde c_l}(z_{i-1},p_{i-1})-{\cal L}_{\tilde c_l}(z_i,p_i)+ \frac{1}{\tilde c_l}\|p_i-p_{i-1}\|^2\right]  \\
	&={\cal L}_{\tilde c_l}(z_{\hat k},p_{\hat k})-{\cal L}_{\tilde c_l}(z_k,p_k)+ \frac{1}{\tilde c_l}\sum_{i={\hat k+1}}^k\|p_i-p_{i-1}\|^2\\
	& \overset{\eqref{auxineq:declemma2}}{=} {\cal L}_{\tilde c_l}(z_{\hat k},p_{\hat k-1})-{\cal L}_{\tilde c_l}(z_{k},p_{k})+ \frac{1}{\tilde c_l}\sum_{i={\hat k}}^k\|p_i-p_{i-1}\|^2 \\
	& \overset{\eqref{eq:auxboundpk2}}{\leq} {\cal L}_{\tilde c_l}(z_{\hat k},p_{\hat k-1})-{\cal L}_{\tilde c_l}(z_{k},p_{k}) - \frac{\|p_k\|^2}{2 \tilde c_l} + \frac{9(k-\hat k)\kappa_p^2}{\tilde c_l} \\
	& =  \left(\sum_{i=\hat k+1}^{k} \lambda_i \right) \Delta_k + \frac{9(k-\hat k)\kappa_p^2}{\tilde c_l},
	\end{align*}
where the last equality follows from the definition of $\Delta_k$  in \eqref{eq:stopCOndition-deltak}.
Now, using the above bound and \eqref{lower lambda bound} we have:
\begin{equation*}
\frac{\sum_{i=\hat k+1}^{k} \lambda_i\|w_i\|^2}{\sum_{i=\hat k+1}^{k} \lambda_i}\leq C_{\sigma}\left(\Delta_k + \frac{9(k-\hat k)\kappa_p^2}{\tilde c_l \sum_{i=\hat k+1}^{k}\lam_i}\right)\overset{\eqref{lower lambda bound}}{\leq}C_{\sigma}\left(\Delta_k + \frac{9\kappa_p^2}{\underline \lambda \tilde c_l}\right).
\end{equation*}
The result follows immediately from the above bound.
\end{proof}
The next result establishes bounds on
$\|Az_k-b\|$ and on
the quantity
$\Delta_k$ defined in \eqref{eq:stopCOndition-deltak}.

\begin{lemma}\label{lem:upperboundDeltak} Consider the sequence of iterates $\{(z_{k},c_{k},p_{k})\}_{k\in {\cal C}_l}$ generated during the $l$-th cycle of  AS-PAL and let $\Delta_k$ be as in \eqref{eq:stopCOndition-deltak}.
Then, for every $k\in {\cal C}_l$,
\begin{itemize}
\item[(a)] we have
\begin{equation}\label{feasibility bound}
\|Az_k-b\|\leq \frac{2\kappa_p}{\tilde c_l};
\end{equation}
\item[(b)]
if additionally $k\geq \hat k+1$, then
\begin{equation}\label{delta k bound}
\Delta_k \le \frac{\kappa_d}{\sum_{i= \hat k+1}^{k}\lambda_i},
\end{equation}
where $\kappa_d$ is as in \eqref{S and kappad} and $\hat k$ denotes the
first index in ${\cal C}_l$.
\end{itemize}
	\end{lemma}
\begin{proof} 
(a)
Let $k \in \cal C_{l}$. Using the update for $p_k$ in \eqref{eq:dual_update2}, triangle inequality, and the bound on $p_k$ in \eqref{ineq:pkbounded}, we have:
\begin{equation*}
\|Az_k-b\|\overset{\eqref{eq:dual_update2}}{=}\frac{\|p_k-p_{k-1}\|}{c_k} \overset{k \in \cal C_l}{\leq}  \frac{\|p_k\|+\|p_{k-1}\|}{\tilde c_{l}}\overset{\eqref{ineq:pkbounded}}{\leq}\frac{2\kappa_p}{\tilde c_l}
\end{equation*}
which immediately proves \eqref{feasibility bound}.

(b) 
Recall from \eqref{def:ctilde-l} that $\cal C_{l} := \{k : c_k=\tilde c_l:= 2^{l-1} c_1 \}$. Then, using the Cauchy-Schwarz inequality,
the definition of the Lagrangian
function in \eqref{lagrangian2}, the definition of $S$ in \eqref{S and kappad}, relations \eqref{ineq:pkbounded} and \eqref{feasibility bound},
and the fact that $\tilde c_l \ge c_1$,
we have
\begin{align}\label{upper bound lagrange2}
\mathcal L_{\tilde c_{l}}(z_{\hat k},p_{\hat k-1})&\le S+\|p_{\hat k-1}\| \|Az_{\hat k}-b\| +\frac{\tilde c_l}{2}\|Az_{\hat k}-b\|^2
\overset{\eqref{feasibility bound}}{\leq}  S+\|p_{\hat k-1}\| \left(\frac{2\kappa_p}{\tilde c_{l}} \right)+\frac{2\kappa^2_p}{\tilde c_{l}}
\overset{\eqref{ineq:pkbounded}}{\leq} S+\frac{4\kappa_p^2}{c_1}.
\end{align}
Let $k \in \mathcal C_l$ be such that $k \geq \hat k+1$. Using the definition of $\phi_{*}$ in \eqref{phi* and lambda bound and bar d} and completing the square, we have:
\begin{equation}\label{lower bound lagrange}
 \cal L_{\tilde c_l}(z_k,p_k) - \phi_*
    \ge \mathcal L_{\tilde c_l}(z_k,p_k) - (f+h)(z_k) 
    = \frac{1}{2}\left\|\frac{p_k}{\sqrt{\tilde c_l}}+\sqrt{\tilde c_l}(Az_k-b)\right\|^2-\frac{\|p_k\|^2}{2\tilde c_l}\geq -\frac{\|p_k\|^2}{2\tilde c_l}.
\end{equation}
Hence, it follows from 
the definition
of $\Delta_k$ in \eqref{eq:stopCOndition-deltak} and relations \eqref{upper bound lagrange2} and \eqref{lower bound lagrange} that
\begin{align*}
\Delta_k = \frac{1}{\sum_{i=\hat k +1}^{k}\lambda_i}\left({\cal L}_{\tilde c_l}(z_{\hat k},p_{\hat k-1})- {\cal L}_{\tilde c_l}(z_k,p_k)-\frac{\|p_k\|^{2}}{2\tilde c_l}\right)\leq \frac{1}{\sum_{i=\hat k+1}^{k}\lambda_i}\left(S+\frac{4\kappa^2_p}{ c_1}-\phi_*\right).
\end{align*}
Thus, \eqref{delta k bound} immediately follows from the definition of $\kappa_d$ in \eqref{S and kappad}.
\end{proof}

The following result establishes
bounds on the number of ACG and outer iterations performed during an AS-PAL cycle and shows that AS-PAL outputs a $(\hat\rho,\hat\eta)$-approximate stationary solution of \eqref{eq:main_prb}
within a logarithmic number of cycles.

\begin{proposition}\label{lem:StaticIPAAL}
The following statements about AS-PAL hold: 
\begin{itemize}
\item[(a)] every cycle  performs 
at most 
\begin{equation}\label{outer iterations bound}
\left\lceil2+\frac{2 C_\sigma\kappa_d}{\underline \lambda \hat \rho^2} \right\rceil
\end{equation}
outer iterations,
where $\underline \lambda$, $\kappa_d$, and $C_\sigma$ are as in \eqref{phi* and lambda bound and bar d}, \eqref{S and kappad}, and \eqref{def:lamb-C1} respectively;
moreover, if $\overline \lambda$ is such that $\overline{\lam} = \Omega(m_f^{-1})$ and $\log^{+}_{0}(m_f\overline \lambda)\leq {\cal O}(1+\kappa_d/(\underline \lambda \hat \rho^2))$,
then the number of ACG calls within an arbitrary cycle is
${\cal O}(1+ m_f \kappa_d/\hat \rho^2)$;
\item[(b)]
for any cycle $l$ of
AS-PAL, its
penalty parameter
satisfies
$\tilde c_{ l} \le \max\{c_1,2\hat{c}\}$
where $\hat c :=\hat c(\hat \rho, \hat \eta)$ and $\hat c(\hat \rho, \hat \eta)$ is as in \eqref{nablaf and kappap and chat};
as a consequence,
the number of cycles
of
AS-PAL
is bounded by
\begin{equation}\label{log number of cycles}
 \log_1^+\left(\frac{4 \hat c}{c_1}\right)
\end{equation}
where $c_1$ is the initial penalty parameter for AS-PAL.
\end{itemize}
\end{proposition}

\begin{proof}
(a) Fix a cycle $l$ and let $\hat k = k_l$ denote the first index in $\mathcal C_{l}$ (see \eqref{def C}). If some $ k\in {\mathcal C}_l$ is such that 
\begin{equation}\label{outer iter bound}
k> \hat{k}+\frac{2C_\sigma\kappa_d}{\underline \lambda \hat \rho^2} 
\end{equation}
then
\begin{equation}
\Delta_k\overset{\eqref{delta k bound}}{\leq}\frac{\kappa_d}{\sum_{i=\hat k+1}^{k}\lam_i}\overset{\eqref{lower lambda bound}}{\leq} \frac{\kappa_d}{\underline \lambda(k-\hat k)}\overset{\eqref{outer iter bound}}{\leq}\frac{\hat \rho^2}{2C_\sigma}
\end{equation}
which clearly implies that $\Delta_k$ satisfies inequality \eqref{eq:stopCOndition-deltak}
and hence that the $l$-th cycle ends at or before the $k$-th iteration. Hence, the first part of (a) follows immediately from this conclusion.
To prove the second part, first note that
the number of times $\lam$ is divided by $2$
in step 2 of AS-PAL is at most $\lceil \log_0^+(\overline \lambda/{\underline \lam})/\log 2 \rceil$,
in view of the last conclusion of Proposition~\ref{ACG facts}(b).
This observation, the conclusion of the
first part, the
two conditions imposed on $\overline \lambda$,
and the definition of $\underline \lam$ in \eqref{phi* and lambda bound and bar d},
then imply that
the number
of ACG calls within an arbitrary cycle
is ${\cal O}(1+ m_f \kappa_d/\hat \rho^2)$.


(b) 
Assume by contradiction that $\tilde c_l >\max\{c_1,2\hat c\}$.
Since $\tilde c_1=c_1$
in view of \eqref{def:ctilde-l},
this implies that
$l>1$ and
$\tilde c_l > 2 \hat c$
and hence that
$\tilde c_{l-1} > \hat c$
in view of the fact that
$\tilde c_l = 2 \tilde c_{l-1}$.
Hence, it follows from
the definition of $\hat c:=\hat c(\hat \rho, \hat \eta)$ in \eqref{nablaf and kappap and chat} and
Lemma~\ref{lem:upperboundDeltak}(a) with $l=l-1$ that
for every $k\in \mathcal C_{l-1}$,
\begin{equation}\label{feasibility bound contradiction}
\|Az_k-b\|\overset{\eqref{feasibility bound}}{\leq} \frac{2\kappa_p}{\tilde c_{l-1}}<\frac{2\kappa_p}{\hat c}< \eta.
\end{equation}
This implies that
$\min_{i \in \cal C_{l-1}} \|w_{i}\|>\hat \rho$
in view of the termination criterion of step 3 and
the fact that AS-PAL has not stopped in the $(l-1)$-th cycle.
Letting $\hat k := k_{l-1}$,
this conclusion
together with Lemma~\ref{lem:minwk-Deltak}
with $l=l-1$ then imply that
\[
\hat \rho^2 <
\frac{\sum_{i=\hat k+1}^{k} \lambda_i\|w_i\|^2}{\sum_{i=\hat k+1}^{k} \lambda_i} \le C_{\sigma}\left(\Delta_k + \frac{9\kappa^2_p}{\underline \lambda \tilde c_{l-1}}\right) < C_{\sigma}\left(\Delta_k + \frac{9\kappa^2_p}{\underline \lambda \hat c}\right) \le C_{\sigma}\Delta_{k} + \frac{\hat \rho^2}{2}
\]
where the third inequality follows from the fact that
$\tilde c_{l-1}>\hat c$ and
the fourth one follows the definition of $\hat c:=\hat c(\hat \rho, \hat \eta)$ in \eqref{nablaf and kappap and chat}.
Using this last conclusion, we can easily see that
\eqref{eq:stopCOndition-deltak} is violated for
 every $k \in {\cal C}_{l-1}$ such that $k \ge \hat k+1$, a conclusion
 that contradicts the fact that the $(l-1)$-th cycle terminated.
\end{proof}
We are now ready to prove Theorem~\ref{theor:StaticIPAAL}.

\begin{proof}[of Theorem~\ref{theor:StaticIPAAL}]
First, note that the assumptions that $\overline{\lam} = \Omega(m_f^{-1})$, $\log^{+}_{0}(m_f\overline \lambda)\leq {\cal O}(1+\kappa_d/(\underline \lambda \hat \rho^2))$, the definition of $\underline \lambda$ in
\eqref{phi* and lambda bound and bar d}, and the second conclusion of Proposition~\ref{lem:StaticIPAAL}(a) imply that every cycle of
AS-PAL performs ${\cal O}(1+ m_f \kappa_d/\hat \rho^2)$ ACG calls.
Second, the assumption that $c_1 \leq 4\hat c$ and Proposition~\ref{lem:StaticIPAAL}(b) imply that
$\tilde c_{ l} \le 4\hat c$ and hence that
$\mathcal M(\tilde c_l) \le \mathcal M(\hat c)$ in view of
the definition of $\mathcal M(c)$ in Theorem~\ref{theor:StaticIPAAL}. The result then immediately follows from the above observations, Proposition~\ref{ACG facts}(a), and the bound \eqref{log number of cycles} on the number of cycles performed by AS-PAL.
\end{proof}

\section{Numerical Experiments} \label{sec:numerical}
This section showcases the numerical performance of AS-PAL, nicknamed ASL, against five other benchmark algorithms for solving five classes of linearly-constrained SNCO problems. It contains five subsections. Each subsection reports the numerical results on a different class of linearly-constrained SNCO problems.

We have implemented a more aggressive variant of ASL, whose details we now describe. First, the variant differs from ASL in that it allows
the prox stepsize to be doubled in step 5 of any iteration if it has not been halved in step 2 and the number of iterations performed by its ACG call in step 1 has not exceeded a pre-specified number.
Second, since the prox stepsize is allowed to increase in this variant, the initial prox stepsize is taken to be relatively small.
Third, our implementation chooses the following values for the input parameters of ASL:
\[
\quad\sigma=0.1, \quad \mu=1/4, \quad \chi=0.5005, \quad \beta=1.25, \quad p_{0}=0.
\]
Finally, for $k \geq 1$, if $L_k$ is the last estimated Lipschitz constant generated by ADAP-FISTA at the end of step 2 of the $k^{\rm th}$ iteration of ASL, then we take $M^{k+1}_0=L_k$.

Now, we describe the implementation details of the five benchmark algorithms which we compare our algorithm with. We consider the iALM method of \cite{ImprovedShrinkingALM20}, two variants of the S-prox-ALM of \cite{ErrorBoundJzhang-ZQLuo2020, ADMMJzhang-ZQLuo2020} (nicknamed SPA1 and SPA2), the inexact proximal augmented Lagrangian method of \cite{RenWilmelo2020iteration} (nicknamed IPL), and the relaxed quadratic penalty method of \cite{WJRComputQPAIPP} (nicknamed RQP). The implementation of iALM chooses the parameters  $\sigma$, $\beta_0$, $w_0$, $\textbf{y}^0$, and $\gamma_k$ as
\[
\sigma=5,\quad\beta_{0}=1 ,\quad w_{0}=1,\quad\boldsymbol{y}^{0}=0,\quad\gamma_{k}=\frac{\left(\log2\right)\|Ax^{1}\|}{(k+1)\left[\log(k+2)\right]^{2}}, 
\]
for every $k \geq 1$. Furthermore, the implementation of iALM uses the ACG subroutine called APG. The starting point for the $k^{\rm th}$ APG call is the prox center for the $k^{\rm th}$ prox subproblem.  The implementations of SPA1 and SPA2 also choose the parameters $\alpha$, $p$, $c$, $\beta$, $y_0$, and $z_0$ as
\[
\alpha=\frac{\Gamma}{4},\quad p=2(L_{f}+\Gamma\|A\|^{2}),\quad c=\frac{1}{2(L_{f}+\Gamma\|A\|^{2})},\quad\beta=0.5,\quad y_{0}=0,\quad z_{0}=x_{0},
\]
where $\Gamma=1$ in SPA1 and $\Gamma=10$ in SPA2. The implementation of IPL sets $\sigma=0.3$, initial penalty parameter $c_1=1$, and constant prox parameter $\lambda=1/(2m_f)$. RQP uses the AIPPv2 variant in \cite{WJRComputQPAIPP} with initial prox stepsize $\lambda=1/m_f$, $\sigma=0.3$, and parameters $(\theta, \tau)=(4,10\left[\lambda L_f+1\right])$. Finally, note that IPL and RQP solve each prox subproblem using the ACG variant in \cite{MontSvaiter_fista} with an adaptive line search for the ACG variant's stepsize parameter as described in \cite{KongThesis2021}.

We describe the type of solution each of the methods aims to find. That is, given a linear operator $A$, functions $f$ and $h$ satisfying assumptions described in Subsection~\ref{subsec:prb_of_interest}, an initial point $z_0 \in \mathcal H$, and tolerance pair $(\hat \rho, \hat \eta)\in\R_{++}^{2}$, each of the methods aims to find a triple $(z,p,w)$ satisfying:
\begin{equation}\label{solution type comp study}
w\in\nabla f(z)+\partial h(z)+A^{*}p, \quad \frac{\|w\|}{1+\|\nabla f(z_{0})\|}\leq \hat\rho ,\quad\frac{\|Az-b\|}{1+\|Az_{0}-b\|}\leq\hat\eta,
\end{equation}
where $\|\cdot\|$ signifies the Euclidean norm when solving vector problems and the Frobenius norm when solving matrix problems. Note that SPA1 and SPA2 are only included for comparison in the experiments of Subsection~\ref{subsec:nonconvex_qpsimplex} and Subsection~\ref{subsec:nonconvex_qpbox} since they are only guaranteed to converge when $h$ is the indicator function of a polyhedron.

The tables below report the runtimes and the total number of ACG iterations needed to find a triple satisfying \eqref{solution type comp study}. The bold numbers in the tables of this section indicate the algorithm that performed the best for that particular metric (i.e. runtime or ACG iterations).
It will be seen in the following subsections that the two adaptive methods ASL
and RQP are the most consistent ones among
the six considered.
More specifically,
within the specified time limit for each problem class,
ASL converged in all instances considered in our experiments while RQP converged in 90\% of them. To compare these two methods on a particular problem class more closely,
we also report in each table caption the following average time ratio (ATR) between
ASL and RQP defined as
\begin{equation}\label{speed comparison}
ATR=\frac{1}{N}\sum_{i=1}^{N}a_i/r_i,   
\end{equation}
where $N$ is the number of class instances that both methods were able to solve and $a_i$ and $r_i$ are the runtimes of ASL and RQP for instance $i$, respectively.

Finally, we note that all experiments were performed in MATLAB 2020a and run on a Macbook Pro with 8-core Intel Core i9 processor and 32 GB of memory. All codes for these experiments are also available online\footnote{See https://github.com/asujanani6/AS-PAL}.

\subsection{Nonconvex QP}\label{subsec:nonconvex_qpsimplex}
Given a pair of dimensions $(\ell,n)\in\mathbb{N}^{2}$, a scalar
pair $(\tau_1, \tau_2)\in\R_{++}^{2}$, matrices $A,C\in\R^{\ell\times n}$
and $B\in\r^{n\times n}$, positive diagonal matrix $D\in\R^{n\times n}$,
and a vector pair $(b,d)\in\R^{\ell}\times\R^{\ell}$, we consider
the problem 
\begin{align*}
\min_{z}\  & \left[f(z):=-\frac{\tau_{1}}{2}\|DBz\|^{2}+\frac{\tau_{2}}{2}\|{ C}z-d\|^{2}\right]\\
\text{s.t.}\  & Az=b,\quad z\in\Delta^{n},
\end{align*}
where $\Delta^{n}:=\{x\in\ \Re_+^{n}:\sum_{i=1}^{n}x_{i}=1\}$.
\begin{table}[!tbh]
\begin{centering}
\begin{tabular}{>{\centering}p{0.6cm}>{\centering}p{0.1cm}|>{\centering}p{0.7cm}>{\centering}p{0.55cm}>{\centering}p{0.55cm}>{\centering}p{0.55cm}>{\centering}p{0.7cm}>{\centering}p{0.9cm}|>{\centering}p{0.9cm}>{\centering}p{0.9cm}>{\centering}p{0.9cm}>{\centering}p{0.95cm}>{\centering}p{0.9cm}>{\centering}p{1.05cm}}
\multicolumn{2}{c|}{\textbf{\scriptsize{}Parameters}} & \multicolumn{6}{c|}{\textbf{\scriptsize{}Iteration Count}} & \multicolumn{6}{c}{\textbf{\scriptsize{}Runtime (seconds)}}\tabularnewline
\hline 
{\scriptsize{}$m_f$} & {\scriptsize{}$L_{f}$} & {\scriptsize{}iALM} & {\scriptsize{}IPL} & {\scriptsize{}RQP} & {\scriptsize{}ASL} & {\scriptsize{}SPA1} & {\scriptsize{}SPA2} & {\scriptsize{}iALM} & {\scriptsize{}IPL} & {\scriptsize{}RQP} & {\scriptsize{}ASL} & {\scriptsize{}SPA1} & {\scriptsize{}SPA2} \tabularnewline
\hline 
\hline 
{\scriptsize{}$10^0$} & {\scriptsize{}$10^1$} & {\scriptsize{}176005} & {\scriptsize{}11202} & \textbf{\scriptsize{}3905} & {\scriptsize{}4762} & {\scriptsize{}*} & {\scriptsize{}*} & {\scriptsize{}4410.98} & {\scriptsize{}373.54} & \textbf{\scriptsize{}111.11} & {\scriptsize{}232.89} & {\scriptsize{}*} & {\scriptsize{}*}\tabularnewline

{\scriptsize{}$10^0$} & {\scriptsize{}$10^2$} & {\scriptsize{}109988} & {\scriptsize{}5382} & {\scriptsize{}6065} & \textbf{\scriptsize{}1884} & {\scriptsize{}*} & {\scriptsize{}*}& {\scriptsize{}1878.44} & {\scriptsize{}155.70} & {\scriptsize{}198.38} & \textbf{\scriptsize{}64.50} & {\scriptsize{}*} & {\scriptsize{}*}\tabularnewline

{\scriptsize{}$10^0$} & {\scriptsize{}$10^3$} & {\scriptsize{}57869} & {\scriptsize{}1210} & {\scriptsize{}11216} & \textbf{\scriptsize{}645}  & {\scriptsize{}*} & {\scriptsize{}*} & {\scriptsize{}1607.23} & {\scriptsize{}55.25} & {\scriptsize{}384.88} & \textbf{\scriptsize{}20.39} & {\scriptsize{}*} & {\scriptsize{}*}\tabularnewline
\hline 

{\scriptsize{}$10^1$} & {\scriptsize{}$10^1$} & {\scriptsize{}236655} & {\scriptsize{}3958} & {\scriptsize{}3171} & \textbf{\scriptsize{}1236} & {\scriptsize{}*} & {\scriptsize{}*} & {\scriptsize{}4785.86} & {\scriptsize{}144.45} & {\scriptsize{}88.10} & \textbf{\scriptsize{}39.36} & {\scriptsize{}*} & {\scriptsize{}*}\tabularnewline

{\scriptsize{}$10^1$} & {\scriptsize{}$10^2$} & {\scriptsize{}195714} & {\scriptsize{}2319} & {\scriptsize{}6701} & \textbf{\scriptsize{}1051} & {\scriptsize{}*} & {\scriptsize{}*} & {\scriptsize{}3582.34} & {\scriptsize{}84.14} & {\scriptsize{}217.43} & \textbf{\scriptsize{}34.07} & {\scriptsize{}*} & {\scriptsize{}*}\tabularnewline

{\scriptsize{}$10^1$} & {\scriptsize{}$10^3$} & {\scriptsize{}98865} & {\scriptsize{}1171}  & {\scriptsize{}7583} & \textbf{\scriptsize{}644}  & {\scriptsize{}*} & {\scriptsize{}*} & {\scriptsize{}2073.41} & {\scriptsize{}41.98} & {\scriptsize{}234.67} & \textbf{\scriptsize{}20.55} & {\scriptsize{}*} & {\scriptsize{}*}\tabularnewline

{\scriptsize{}$10^1$} & {\scriptsize{}$10^4$} & {\scriptsize{}87595} & {\scriptsize{}6506}  & {\scriptsize{}15637} & \textbf{\scriptsize{}924} & {\scriptsize{}*} & {\scriptsize{}*} & {\scriptsize{}3272.03} & {\scriptsize{}280.97} & {\scriptsize{}403.79} & \textbf{\scriptsize{}29.51} & {\scriptsize{}*} & {\scriptsize{}*}\tabularnewline
\hline 

{\scriptsize{}$10^2$} & {\scriptsize{}$10^3$} & {\scriptsize{}366178} & {\scriptsize{}*} & {\scriptsize{}7647} & \textbf{\scriptsize{}778} & {\scriptsize{}92872} & {\scriptsize{}*} & {\scriptsize{}6637.79} & {\scriptsize{}*} &  {\scriptsize{}207.66} & \textbf{\scriptsize{}25.22} & {\scriptsize{}3290.91} & {\scriptsize{}*}\tabularnewline

{\scriptsize{}$10^2$} & {\scriptsize{}$10^4$} & {\scriptsize{}248673} & {\scriptsize{}*} & {\scriptsize{}10421} & \textbf{\scriptsize{}1375} & {\scriptsize{}120882} & {\scriptsize{}257973} & {\scriptsize{}4329.35} & {\scriptsize{}*} & {\scriptsize{}283.96} & \textbf{\scriptsize{}45.27} & {\scriptsize{}5363.80} & {\scriptsize{}10644.96}\tabularnewline

{\scriptsize{}$10^2$} & {\scriptsize{}$10^5$} & {\scriptsize{}130351} & {\scriptsize{}19887} & {\scriptsize{}16250} & \textbf{\scriptsize{}2410} & {\scriptsize{}205483} & {\scriptsize{}213369} & {\scriptsize{}2310.50} & {\scriptsize{}561.16} & {\scriptsize{}447.53} & \textbf{\scriptsize{}80.98} & {\scriptsize{}9317.19} & {\scriptsize{}7548.79}\tabularnewline
\hline

{\scriptsize{}$10^3$} & {\scriptsize{}$10^3$} & {\scriptsize{}363915} & {\scriptsize{}*} & {\scriptsize{}4589} & \textbf{\scriptsize{}2001}  & {\scriptsize{}*} & {\scriptsize{}*} & {\scriptsize{}8111.85} & {\scriptsize{}*}  & {\scriptsize{}136.56} & \textbf{\scriptsize{}71.45} & {\scriptsize{}*} & {\scriptsize{}*}\tabularnewline

{\scriptsize{}$10^3$} & {\scriptsize{}$10^4$} & {\scriptsize{}344723} & {\scriptsize{}*} & {\scriptsize{}6023} & \textbf{\scriptsize{}4055} & {\scriptsize{}*} & {\scriptsize{}158622} & {\scriptsize{}6949.95} & {\scriptsize{}*} & {\scriptsize{}596.66} & \textbf{\scriptsize{}168.01} & {\scriptsize{}*} & {\scriptsize{}6136.37} \tabularnewline

{\scriptsize{}$10^3$} & {\scriptsize{}$10^5$} & {\scriptsize{}291006} & {\scriptsize{}16455} & {\scriptsize{}10067} & \textbf{\scriptsize{}3007} & {\scriptsize{}*} & {\scriptsize{}286333}  & {\scriptsize{}5714.73} & {\scriptsize{}495.64} & {\scriptsize{}279.27} & \textbf{\scriptsize{}107.44} & {\scriptsize{}*} & {\scriptsize{}10761.87} \tabularnewline

{\scriptsize{}$10^3$} & {\scriptsize{}$10^6$} & {\scriptsize{}141115} & {\scriptsize{}21586} & {\scriptsize{}15991} & \textbf{\scriptsize{}2208} & {\scriptsize{}269687} & {\scriptsize{}175752} & {\scriptsize{}2527.15} & {\scriptsize{}610.60} & {\scriptsize{}423.22} & \textbf{\scriptsize{}89.73} & {\scriptsize{}9718.92} & {\scriptsize{}6267.26} \tabularnewline
\end{tabular}
\par\end{centering}
\caption{Iteration counts and runtimes (in seconds) for the Nonconvex QP
problem in Subsection~\ref{subsec:nonconvex_qpsimplex}. The tolerances are set to $10^{-4}$. Entries marked with * did not converge in the time limit of $10800$ seconds. The ATR metric is $0.3644$. \label{tab104:qpsimplex}}
\end{table}

\begin{table}[!tbh]
\begin{centering}
\begin{tabular}{>{\centering}p{0.7cm}>{\centering}p{0.8cm}|>{\centering}p{1.4cm}>{\centering}p{1.4cm}>{\centering}p{1.2cm}|>{\centering}p{1.4cm}>{\centering}p{1.4cm}>{\centering}p{1.4cm}}
\multicolumn{2}{c|}{\textbf{\scriptsize{}Parameters}} & \multicolumn{3}{c|}{\textbf{\scriptsize{}Iteration Count}} & \multicolumn{3}{c}{\textbf{\scriptsize{}Runtime (seconds)}}\tabularnewline
\hline 
{\normalsize{}$m_f$} & {\normalsize{}$L_{f}$} & {\normalsize{}iALM} & {\normalsize{}RQP} & {\normalsize{}ASL}& {\normalsize{}iALM} & {\normalsize{}RQP} & {\normalsize{}ASL}\tabularnewline
\hline 
\hline 
{\normalsize{}$10^0$} & {\normalsize{}$10^1$} & {\normalsize{}591803} & {\normalsize{}23935} & \textbf{\normalsize{}8276} & {\normalsize{}9779.56} & {\normalsize{}599.52} & \textbf{\normalsize{}419.69}\tabularnewline

{\normalsize{}$10^0$} & {\normalsize{}$10^2$} & {\normalsize{}698270} & {\normalsize{}62409} & \textbf{\normalsize{}2474} & {\normalsize{}11336.43} & {\normalsize{}1579.43} & \textbf{\normalsize{}87.09}\tabularnewline

{\normalsize{}$10^0$} & {\normalsize{}$10^3$} & {\normalsize{}551623} & {\normalsize{}84314} & \textbf{\normalsize{}959} & {\normalsize{}9146.99} & {\normalsize{}2232.40} & \textbf{\normalsize{}31.16}\tabularnewline
\hline 

{\normalsize{}$10^1$} & {\normalsize{}$10^1$} & {\normalsize{}*} & {\normalsize{}25312} & \textbf{\normalsize{}1628} & {\normalsize{}*} & {\normalsize{}703.17} & \textbf{\normalsize{}66.27}\tabularnewline

{\normalsize{}$10^1$} & {\normalsize{}$10^2$} & {\normalsize{}*} & {\normalsize{}53161} & \textbf{\normalsize{}1793} & {\normalsize{}*} & {\normalsize{}3386.99} & \textbf{\normalsize{}77.04}\tabularnewline

{\normalsize{}$10^1$} & {\normalsize{}$10^3$} & {\normalsize{}*} & {\normalsize{}54172} & \textbf{\normalsize{}927} & {\normalsize{}*} & {\normalsize{}1438.63} & \textbf{\normalsize{}34.01}\tabularnewline

{\normalsize{}$10^1$} & {\normalsize{}$10^4$} & {\normalsize{}*} & {\normalsize{}108376} & \textbf{\normalsize{}1477} & {\normalsize{}*} & {\normalsize{}3482.75} & \textbf{\normalsize{}79.91}\tabularnewline
\hline 

{\normalsize{}$10^2$} & {\normalsize{}$10^3$} & {\normalsize{}*} & {\normalsize{}92292} & \textbf{\normalsize{}1251} & {\normalsize{}*} & {\normalsize{}2475.48} & \textbf{\normalsize{}61.80}\tabularnewline

{\normalsize{}$10^2$} & {\normalsize{}$10^4$} & {\normalsize{}*} & {\normalsize{}78775} & \textbf{\normalsize{}1992} & {\normalsize{}*} & {\normalsize{}2116.42} & \textbf{\normalsize{}110.03}\tabularnewline

{\normalsize{}$10^2$} & {\normalsize{}$10^5$} & {\normalsize{}*} & {\normalsize{}137886} & \textbf{\normalsize{}3940} & {\normalsize{}*} & {\normalsize{}3875.34} & \textbf{\normalsize{}219.18}\tabularnewline
\hline

{\normalsize{}$10^3$} & {\normalsize{}$10^3$} & {\normalsize{}*} & {\normalsize{}47491} & \textbf{\normalsize{}2238} & {\normalsize{}*} & {\normalsize{}1280.58} & \textbf{\normalsize{}130.45}\tabularnewline

{\normalsize{}$10^3$} & {\normalsize{}$10^4$} & {\normalsize{}*} & {\normalsize{}49708} & \textbf{\normalsize{}6035} & {\normalsize{}*} & {\normalsize{}596.66} & \textbf{\normalsize{}168.01}\tabularnewline

{\normalsize{}$10^3$} & {\normalsize{}$10^5$} & {\normalsize{}*} & {\normalsize{}52883} & \textbf{\normalsize{}3863} & {\normalsize{}*} & {\normalsize{}1481.81} & \textbf{\normalsize{}220.99}\tabularnewline

{\normalsize{}$10^3$} & {\normalsize{}$10^6$} & {\normalsize{}*} & {\normalsize{}108743} & \textbf{\normalsize{}3396} & {\normalsize{}*} & {\normalsize{}4083.58} & \textbf{\normalsize{}179.08}\tabularnewline
\end{tabular}
\par\end{centering}
\caption{Iteration counts and runtimes (in seconds) for the Nonconvex QP
problem in Subsection~\ref{subsec:nonconvex_qpsimplex}. The tolerances are set to $10^{-6}$. Entries marked with * did not converge in the time limit of $21600$ seconds. The ATR metric is $0.1173$.\label{tab106:qpsimplex}}
\end{table}

For our experiments in this subsection, we choose dimensions $(l,n)=(20,1000)$ and generate the matrices $A$, $B$, and $C$ to be fully dense. The entries
of $A$, $B$, $C$, and $d$ (resp. $D$) are generated by sampling
from the uniform distribution ${\cal U}[0,1]$ (resp. ${\cal U}[1,1000]$). We generate the vector $b$ as $b=A(\boldsymbol{e}/n)$ where $\boldsymbol{e}$
denotes the vector of all ones. The initial starting point $z_{0}$
is generated as $z^{*}/\sum_{i=1}^{n}z^{*}_{i}$, where the
entries of $z^{*}$ are sampled from the ${\cal U}[0,1]$ distribution. Finally, we choose $(\tau_1, \tau_2)\in\R_{++}^{2}$ so that $L_{f}=\lam_{\max}(\nabla^{2}f)$
and $m_{f}=-\lam_{\min}(\nabla^{2}f)$ are the various values given in the tables of this subsection.

We now describe the specific parameters that ASL, RQP, and iALM choose for this class of problems. Both ASL and RQP choose the initial penalty parameter, $c_1=1$. ASL allows the prox stepsize to be doubled at the end of an iteration if the number of iterations by its ACG call does not exceed $75$. ASL also takes $M^1_0$ defined in its step 1 to be 100 and the initial prox stepsize to be $20/m_f$. Finally, the auxillary parameters of iALM are given by:\[B_i=\|a_i\|, \quad  L_i=0, \quad \rho_i=0 \quad \forall i \geq 1,\]
where $a_i$ is the $i^{\rm th}$ row of $A$. 

The numerical results are presented in two tables, Table~\ref{tab104:qpsimplex} and Table~\ref{tab106:qpsimplex}. The first table, Table~\ref{tab104:qpsimplex}, compares ASL with all five benchmark algorithms namely, iALM, IPL, RQP, SPA1, and SPA2. The tolerances are set as $\hat \rho= \hat \eta=10^{-4}$ and a time limit of $10800$ seconds, or 3 hours, is imposed. Table~\ref{tab106:qpsimplex} presents the same exact instances as Table~\ref{tab104:qpsimplex} but now with tolerances set as $\hat \rho=\hat \eta=10^{-6}$ and a time limit of $21600$ seconds, or $6$ hours. Table~\ref{tab106:qpsimplex} only compares ASL with iALM and RQP since these were the only two other algorithms to converge for every instance with tolerances set at $10^{-4}$. Entries marked with * did not converge in the time limit.

\subsection{Nonconvex QP with Box Constraints}\label{subsec:nonconvex_qpbox}
\begin{table}[!tbh]
\begin{centering}
\begin{tabular}{>{\centering}p{0.1cm}>{\centering}p{0.3cm}>{\centering}p{0.3cm}|>{\centering}p{0.75cm}>{\centering}p{0.6cm}>{\centering}p{0.75cm}>{\centering}p{0.6cm}>{\centering}p{0.8cm}>{\centering}p{0.9cm}|>{\centering}p{0.75cm}>{\centering}p{0.7cm}>{\centering}p{0.8cm}>{\centering}p{0.6cm}>{\centering}p{0.7cm}>{\centering}p{1.0cm}}
\multicolumn{3}{c|}{\textbf{\scriptsize{}Parameters}} & \multicolumn{6}{c|}{\textbf{\scriptsize{}Iteration Count}} & \multicolumn{6}{c}{\textbf{\scriptsize{}Runtime (seconds)}}\tabularnewline
\hline 
{\scriptsize{}$r$} & {\scriptsize{}$m_f$} & {\scriptsize{}$L_{f}$} & {\scriptsize{}iALM} & {\scriptsize{}IPL} & {\scriptsize{}RQP} & {\scriptsize{}ASL}& {\scriptsize{}SPA1} & {\scriptsize{}SPA2} & {\scriptsize{}iALM} & {\scriptsize{}IPL} & {\scriptsize{}RQP} & {\scriptsize{}ASL} & {\scriptsize{}SPA1}& {\scriptsize{}SPA2}\tabularnewline
\hline 
\hline 
{\scriptsize{}5} & {\scriptsize{}$10^0$} & {\scriptsize{}$10^1$} & {\scriptsize{}203310} & {\scriptsize{}11274} & {\scriptsize{}49512} & \textbf{\scriptsize{}7247}& {\scriptsize{}205576} & {\scriptsize{}1943184} & {\scriptsize{}226.93} & {\scriptsize{}17.93} & {\scriptsize{}92.43} & \textbf{\scriptsize{}1.69} & {\scriptsize{}335.91}& {\scriptsize{}2879.25}\tabularnewline

{\scriptsize{}10} & {\scriptsize{}$10^0$} & {\scriptsize{}$10^1$} & {\scriptsize{}221433} & {\scriptsize{}9170} & {\scriptsize{}70736}  & \textbf{\scriptsize{}7043}& {\scriptsize{}128567} & {\scriptsize{}1176352} & {\scriptsize{}334.57} & {\scriptsize{}14.29} & {\scriptsize{}132.59} & \textbf{\scriptsize{}1.76} & {\scriptsize{}240.07}& {\scriptsize{}2139.45}\tabularnewline

{\scriptsize{}20} & {\scriptsize{}$10^0$} & {\scriptsize{}$10^1$} & {\scriptsize{}192970} &{\scriptsize{}8363} & {\scriptsize{}58980} & \textbf{\scriptsize{}5469}& {\scriptsize{}154035} & {\scriptsize{}1403641} & {\scriptsize{}307.75} & {\scriptsize{}14.59} & {\scriptsize{}144.06} & \textbf{\scriptsize{}1.31} & {\scriptsize{}374.16}& {\scriptsize{}2295.86}\tabularnewline

\hline

{\scriptsize{}1} & {\scriptsize{}$10^1$} & {\scriptsize{}$10^2$} & {\scriptsize{}465159} &{\scriptsize{}*} & {\scriptsize{}326336} & \textbf{\scriptsize{}4509}& {\scriptsize{}133522} & {\scriptsize{}303003} & {\scriptsize{}858.38} & {\scriptsize{}*} & {\scriptsize{}1156.69} & \textbf{\scriptsize{}1.17} & {\scriptsize{}213.68}& {\scriptsize{}524.57}\tabularnewline

{\scriptsize{}2} & {\scriptsize{}$10^1$} & {\scriptsize{}$10^2$} & {\scriptsize{}862136} & {\scriptsize{}*} & {\scriptsize{}399982} & \textbf{\scriptsize{}8453}& {\scriptsize{}64280} & {\scriptsize{}447451} & {\scriptsize{}1141.23} & {\scriptsize{}*} & {\scriptsize{}814.19} & \textbf{\scriptsize{}2.01} & {\scriptsize{}107.55}& {\scriptsize{}693.07}\tabularnewline

{\scriptsize{}5} & {\scriptsize{}$10^1$} & {\scriptsize{}$10^2$} & {\scriptsize{}1857919} &{\scriptsize{}*} & {\scriptsize{}174005} & \textbf{\scriptsize{}8320}& {\scriptsize{}106715} & {\scriptsize{}488965} & {\scriptsize{}2476.33} & {\scriptsize{}*} & {\scriptsize{}394.47} &  \textbf{\scriptsize{}2.11} & {\scriptsize{}238.12}& {\scriptsize{}879.75}\tabularnewline

\hline
{\scriptsize{}1} & {\scriptsize{}$10^1$} & {\scriptsize{}$10^3$} & {\scriptsize{}351468} &{\scriptsize{}*} & {\scriptsize{}47007} & \textbf{\scriptsize{}8438}& {\scriptsize{}47583} & {\scriptsize{}123195} & {\scriptsize{}510.028} & {\scriptsize{}*} & {\scriptsize{}81.74} & \textbf{\scriptsize{}2.00} & {\scriptsize{}65.33}& {\scriptsize{}166.48}\tabularnewline

{\scriptsize{}2} & {\scriptsize{}$10^1$} & {\scriptsize{}$10^3$} & {\scriptsize{}368578} & {\scriptsize{}*} & {\scriptsize{}69875} & \textbf{\scriptsize{}6200}& {\scriptsize{}96971} & {\scriptsize{}161433} & {\scriptsize{}481.14} & {\scriptsize{}*} & {\scriptsize{}129.77} & \textbf{\scriptsize{}1.58} & {\scriptsize{}123.39}& {\scriptsize{}198.84}\tabularnewline

{\scriptsize{}5} & {\scriptsize{}$10^1$} & {\scriptsize{}$10^3$} & {\scriptsize{}280346} & {\scriptsize{}*} & {\scriptsize{}116988} & \textbf{\scriptsize{}5218}& {\scriptsize{}272448} & {\scriptsize{}161327} & {\scriptsize{}329.16} & {\scriptsize{}*} & {\scriptsize{}232.13} & \textbf{\scriptsize{}1.24} & {\scriptsize{}361.41}& {\scriptsize{}216.67}\tabularnewline

\hline 

{\scriptsize{}1} & {\scriptsize{}$10^2$} & {\scriptsize{}$10^3$} & {\scriptsize{}727587} & {\scriptsize{}*} & {\scriptsize{}104411} & \textbf{\scriptsize{}4200}& {\scriptsize{}*} & {\scriptsize{}112604} & {\scriptsize{}908.05} & {\scriptsize{}*} & {\scriptsize{}205.03} & \textbf{\scriptsize{}1.15} & {\scriptsize{}*}& {\scriptsize{}154.60}\tabularnewline

{\scriptsize{}2} & {\scriptsize{}$10^2$} & {\scriptsize{}$10^3$} & {\scriptsize{}964734} & {\scriptsize{}21472} &  {\scriptsize{}130903} & \textbf{\scriptsize{}6432}& {\scriptsize{}*} & {\scriptsize{}53266} & {\scriptsize{}1225.22} & {\scriptsize{}44.19} & {\scriptsize{}253.02} & \textbf{\scriptsize{}1.56} & {\scriptsize{}*}& {\scriptsize{}74.85}\tabularnewline

{\scriptsize{}5} & {\scriptsize{}$10^2$} & {\scriptsize{}$10^3$} & {\scriptsize{}705884} & {\scriptsize{}11709} & {\scriptsize{}117945} & \textbf{\scriptsize{}5137}& {\scriptsize{}*} & {\scriptsize{}47237} & {\scriptsize{}890.93} & {\scriptsize{}25.93} & {\scriptsize{}226.21} & \textbf{\scriptsize{}1.29} & {\scriptsize{}*}& {\scriptsize{}65.84}\tabularnewline

\hline 
{\scriptsize{}1} & {\scriptsize{}$10^2$} & {\scriptsize{}$10^4$} & {\scriptsize{}576627} & {\scriptsize{}255622} & {\scriptsize{}100193} & \textbf{\scriptsize{}7796}& {\scriptsize{}155586} & {\scriptsize{}183307} & {\scriptsize{}864.79} & {\scriptsize{}575.34} & {\scriptsize{}200.17} & \textbf{\scriptsize{}1.98} & {\scriptsize{}232.28}& {\scriptsize{}274.50}\tabularnewline

{\scriptsize{}2} & {\scriptsize{}$10^2$} & {\scriptsize{}$10^4$} & {\scriptsize{}1028921} & {\scriptsize{}29123} & {\scriptsize{}165257} & \textbf{\scriptsize{}7048}& {\scriptsize{}158192} & {\scriptsize{}196930} & {\scriptsize{}1477.99} & {\scriptsize{}57.82} & {\scriptsize{}314.62} & \textbf{\scriptsize{}1.79} & {\scriptsize{}256.01}& {\scriptsize{}308.35}\tabularnewline

{\scriptsize{}5} & {\scriptsize{}$10^2$} & {\scriptsize{}$10^4$} & {\scriptsize{}652822} & {\scriptsize{}65523} & {\scriptsize{}86597} & \textbf{\scriptsize{}9471}& {\scriptsize{}144334} & {\scriptsize{}181157} & {\scriptsize{}1032.59} & {\scriptsize{}116.36} & {\scriptsize{}169.88} & \textbf{\scriptsize{}2.22} & {\scriptsize{}223.96}& {\scriptsize{}274.79}\tabularnewline

\hline 
{\scriptsize{}5} & {\scriptsize{}$10^3$} & {\scriptsize{}$10^3$} & {\scriptsize{}*} &{\scriptsize{}142961} & {\scriptsize{}225865} & \textbf{\scriptsize{}26333}& {\scriptsize{}*} & {\scriptsize{}*} & {\scriptsize{}*} & {\scriptsize{}253.35} & {\scriptsize{}439.62} & \textbf{\scriptsize{}5.93} & {\scriptsize{}*}& {\scriptsize{}*}\tabularnewline

{\scriptsize{}10} & {\scriptsize{}$10^3$} & {\scriptsize{}$10^3$} & {\scriptsize{}2474551}  & {\scriptsize{}*} & {\scriptsize{}168397} & \textbf{\scriptsize{}14213}& {\scriptsize{}*} & {\scriptsize{}*} & {\scriptsize{}3522.28} & {\scriptsize{}*} & {\scriptsize{}330.62} & \textbf{\scriptsize{}3.27} & {\scriptsize{}*}& {\scriptsize{}*}\tabularnewline

\hline

{\scriptsize{}1} & {\scriptsize{}$10^3$} & {\scriptsize{}$10^4$} & {\scriptsize{}435881}  & {\scriptsize{}71369} & {\scriptsize{}*} & \textbf{\scriptsize{}4724}& {\scriptsize{}*} & {\scriptsize{}*} & {\scriptsize{}667.19} & {\scriptsize{}154.75} & {\scriptsize{}*} & \textbf{\scriptsize{}1.21} & {\scriptsize{}*}& {\scriptsize{}*}\tabularnewline

{\scriptsize{}2} & {\scriptsize{}$10^3$} & {\scriptsize{}$10^4$} & {\scriptsize{}476462}  & {\scriptsize{}23931} & {\scriptsize{}64649} & \textbf{\scriptsize{}8971}& {\scriptsize{}*} & {\scriptsize{}*} & {\scriptsize{}584.73} & {\scriptsize{}39.52} & {\scriptsize{}100.27} & \textbf{\scriptsize{}2.21} & {\scriptsize{}*}& {\scriptsize{}*}\tabularnewline

{\scriptsize{}5} & {\scriptsize{}$10^3$} & {\scriptsize{}$10^4$} & {\scriptsize{}521072}  & {\scriptsize{}9829} & {\scriptsize{}*} & \textbf{\scriptsize{}5943}& {\scriptsize{}*} & {\scriptsize{}*} & {\scriptsize{}649.28} & {\scriptsize{}17.02} & {\scriptsize{}*} & \textbf{\scriptsize{}1.51} & {\scriptsize{}*}& {\scriptsize{}*}\tabularnewline

\hline

{\scriptsize{}1} & {\scriptsize{}$10^3$} & {\scriptsize{}$10^5$} & {\scriptsize{}*}  & {\scriptsize{}347105} & {\scriptsize{}*} & \textbf{\scriptsize{}8952}& {\scriptsize{}*} & {\scriptsize{}142702} & {\scriptsize{}*} & {\scriptsize{}696.61} & {\scriptsize{}*} & \textbf{\scriptsize{}2.18} & {\scriptsize{}*}& {\scriptsize{}231.41}\tabularnewline

{\scriptsize{}2} & {\scriptsize{}$10^3$} & {\scriptsize{}$10^5$} & {\scriptsize{}1436029}  & {\scriptsize{}*} & {\scriptsize{}*} & \textbf{\scriptsize{}9013}& {\scriptsize{}*} & {\scriptsize{}163317} & {\scriptsize{}2222.25} & {\scriptsize{}*} & {\scriptsize{}*} & \textbf{\scriptsize{}2.20} & {\scriptsize{}*}& {\scriptsize{}397.06}\tabularnewline

{\scriptsize{}5} & {\scriptsize{}$10^3$} & {\scriptsize{}$10^5$} & {\scriptsize{}*}  & {\scriptsize{}106935} & {\scriptsize{}*} & \textbf{\scriptsize{}11629}& {\scriptsize{}*} & {\scriptsize{}145047} & {\scriptsize{}*} & {\scriptsize{}276.73} & {\scriptsize{}*} & \textbf{\scriptsize{}2.81} & {\scriptsize{}*}& {\scriptsize{}192.72}\tabularnewline
\end{tabular}
\par\end{centering}
\caption{Iteration counts and runtimes (in seconds) for the Nonconvex QP
problem with box constraints in Subsection~\ref{subsec:nonconvex_qpbox}. The tolerances are set to $10^{-5}$. Entries marked with * did not converge in the time limit of $3600$ seconds. The ATR metric is $0.0102$. \label{tab105:qpbox}}
\end{table}
Given a pair of dimensions $(\ell,n)\in\mathbb{N}^{2}$, a scalar
triple $(r,\tau_1, \tau_2)\in\R_{++}^{3}$, matrices $A,C\in\R^{\ell\times n}$
and $B\in\r^{n\times n}$, positive diagonal matrix $D\in\R^{n\times n}$,
and a vector pair $(b,d)\in\R^{\ell}\times\R^{\ell}$, we consider
the problem 
\begin{align*}
\min_{z}\  & \left[f(z):=-\frac{\tau_{1}}{2}\|DBz\|^{2}+\frac{\tau_{2}}{2}\|{ C}z-d\|^{2}\right]\\
\text{s.t.}\  & Az=b,\\
& -r\leq z_i \leq r, \quad i\in\{1,...,n\}.
\end{align*}

For our experiments in this subsection, we choose dimensions $(l,n)=(20,100)$ and generate the matrices $A$, $B$, and $C$ to be fully dense. The entries
of $A$, $B$, $C$, and $d$ (resp. $D$) are generated by sampling
from the uniform distribution ${\cal U}[0,1]$ (resp. ${\cal U}[1,1000]$). We generate the vector $b$ as $b=A(u)$ where $u$
is a random vector in ${\cal U}[-r,r]^{n}$. The initial starting point $z_{0}$
is generated as a random vector in ${\cal U}[-r,r]^{n}$. We vary $r$ across the different instances. Finally, we choose $(\tau_1, \tau_2)\in\R_{++}^{2}$ so that $L_{f}=\lam_{\max}(\nabla^{2}f)$
and $m_{f}=-\lam_{\min}(\nabla^{2}f)$ are the various values given in the tables of this subsection. 

We now describe the specific parameters that ASL and RQP choose for this class of problems. Both ASL and RQP choose the initial penalty parameter, $c_1=1$. ASL also allows the prox stepsize to be doubled at the end of an iteration if the number of iterations by its ACG call does not exceed $75$. Finally, ASL takes $M^1_0$ defined in its step 1 to be 100 and the initial prox stepsize to be $20/m_f$. 

The numerical results are presented in Table~\ref{tab105:qpbox}. Table~\ref{tab105:qpbox} compares ASL with all five of the benchmark algorithms namely, iALM, IPL, RQP, SPA1, and SPA2. The tolerances are set as $\hat \rho= \hat \eta=10^{-5}$ and a time limit of $3600$ seconds, or 1 hour, is imposed. Entries marked with * did not converge in the time limit.

\subsection{Nonconvex QSDP}\label{subsec:nonconvex_sdpsimplex}
Given a pair of dimensions $(\ell,n)\in\mathbb{N}^{2}$, a scalar
pair $(\tau_1,\tau_2)\in\R_{++}^{2}$, linear operators ${\cal {\cal A}}:\mathbb{S}_{+}^{n}\mapsto\R^{\ell}$,
${\cal {\cal B}}:\mathbb{S}_{+}^{n}\mapsto\R^{n}$, and ${\cal {\cal C}}:\mathbb{S}_{+}^{n}\mapsto\R^{\ell}$
defined by 
\[
\left[{\cal A}(Z)\right]_{i}=\left\langle A_{i},Z\right\rangle ,\quad\left[{\cal B}(Z)\right]_{j}=\left\langle B_{j},Z\right\rangle ,\quad\left[{\cal C}(Z)\right]_{i}=\left\langle C_{i},Z\right\rangle ,
\]
for matrices $\{A_{i}\}_{i=1}^{\ell},\{B_{j}\}_{j=1}^{n},\{C_{i}\}_{i=1}^{\ell}\subseteq\R^{n\times n}$,
positive diagonal matrix $D\in\R^{n\times n}$, and a vector pair
$(b,d)\in\R^{\ell}\times\R^{\ell}$, we consider the
following nonconvex quadratic semidefinite programming  (QSDP) problem: 
\begin{align*}
\min_{Z}\  & \left[f(Z):=-\frac{\tau_{1}}{2}\|D{\cal B}(Z)\|^{2}+\frac{\tau_{2}}{2}\|{\cal C}(Z)-d\|^{2}\right]\\
\text{s.t.}\  & {\cal A}(Z)=b,\quad Z\in P^{n},
\end{align*}
where $P^{n}=\{Z\in\mathbb{S}_{+}^{n}:{\rm trace}\,(Z)=1\}$.
\begin{table}[!tbh]
\begin{centering}
\begin{tabular}{>{\centering}p{0.6cm}>{\centering}p{0.6cm}|>{\centering}p{1.1cm}>{\centering}p{1.1cm}>{\centering}p{1.1cm}>{\centering}p{1.1cm}|>{\centering}p{1.1cm}>{\centering}p{1.1cm}>{\centering}p{1.1cm}>{\centering}p{1.1cm}}
\multicolumn{2}{c|}{\textbf{\scriptsize{}Parameters}} & \multicolumn{4}{c|}{\textbf{\scriptsize{}Iteration Count}} & \multicolumn{4}{c}{\textbf{\scriptsize{}Runtime (seconds)}}\tabularnewline
\hline 
{\small{}$m_f$} & {\small{}$L_{f}$} & {\small{}iALM} & {\small{}IPL} & {\small{}RQP} & {\small{}ASL} & {\small{}iALM} & {\small{}IPL} & {\small{}RQP} & {\small{}ASL}\tabularnewline
\hline 
\hline 
{\small{}$10^0$} & {\small{}$10^1$} & {\small{}230272} & {\small{}27345} & {\small{}9887} & \textbf{\small{}9647} & {\small{}1772.75} & {\small{}322.82} & {\small{}91.62} & \textbf{\small{}85.38}\tabularnewline

{\small{}$10^0$} & {\small{}$10^2$} & {\small{}91421} & {\small{}4575} & {\small{}7085} & \textbf{\small{}1498} & {\small{}516.42} & {\small{}53.13} & {\small{}73.65} & \textbf{\small{}13.28}\tabularnewline

{\small{}$10^0$} & {\small{}$10^3$} & {\small{}113405} & {\small{}1403} & {\small{}9486} & \textbf{\small{}960} & {\small{}587.24} & {\small{}19.54} & {\small{}112.42} & \textbf{\small{}8.32}\tabularnewline

{\small{}$10^0$} & {\small{}$10^4$} & {\small{}393953} & {\small{}3140} & {\small{}10019} & \textbf{\small{}1824} & {\small{}1794.35} & {\small{}31.54} & {\small{}91.49} & \textbf{\small{}15.98}\tabularnewline

{\small{}$10^0$} & {\small{}$10^5$} & {\small{}1938432} & {\small{}16282} & {\small{}15719} & \textbf{\small{}9883} & {\small{}9473.18} & {\small{}166.02} & {\small{}145.12} & \textbf{\small{}85.50}\tabularnewline

\hline
{\small{}$10^1$} & {\small{}$10^2$} & {\small{}347506} & {\small{}*} & {\small{}15971} & \textbf{\small{}4417} & {\small{}1556.07} & {\small{}*} & {\small{}140.53} & \textbf{\small{}38.31}\tabularnewline

{\small{}$10^1$} & {\small{}$10^3$} & {\small{}177264} & {\small{}*} & {\small{}10945} & \textbf{\small{}2151} & {\small{}750.97} & {\small{}*} & {\small{}96.93} & \textbf{\small{}18.98}\tabularnewline

{\small{}$10^1$} & {\small{}$10^4$} & {\small{}129617} & {\small{}1296} & {\small{}9838} & \textbf{\small{}1273} & {\small{}2008.28} & {\small{}15.58} & {\small{}93.88} & \textbf{\small{}11.14}\tabularnewline

{\small{}$10^1$} & {\small{}$10^5$} & {\small{}287924} & {\small{}3410} & {\small{}8040} & \textbf{\small{}2262} & {\small{}1305.24} & {\small{}35.51} & {\small{}75.28} & \textbf{\small{}19.91}\tabularnewline

{\small{}$10^1$} & {\small{}$10^6$} & {\small{}1473676} & {\small{}15855} & {\small{}12696} & \textbf{\small{}10305} & {\small{}7865.52} & {\small{}164.07} & {\small{}120.96} & \textbf{\small{}92.97}\tabularnewline

\hline

{\small{}$10^2$} & {\small{}$10^4$} & {\small{}182388} & {\small{}*} & {\small{}10803} & \textbf{\small{}1261} & {\small{}844.88} & {\small{}*} & {\small{}99.14} & \textbf{\small{}11.55}\tabularnewline

{\small{}$10^2$} & {\small{}$10^6$} & {\small{}450561} & {\small{}4503} & {\small{}10804} & \textbf{\small{}2990} & {\small{}2612.55} & {\small{}59.18} & {\small{}128.30} & \textbf{\small{}25.63}\tabularnewline

{\small{}$10^2$} & {\small{}$10^7$} & {\small{}1034041} & {\small{}20235} & {\small{}14622} & \textbf{\small{}11893} & {\small{}4612.17} & {\small{}207.29} & {\small{}137.38} & \textbf{\small{}104.67}\tabularnewline

\hline
{\small{}$10^3$} & {\small{}$10^4$} & {\small{}552738} & {\small{}*} & {\small{}18530} & \textbf{\small{}1368} & {\small{}2435.47} & {\small{}*} & {\small{}172.45} & \textbf{\small{}12.42}\tabularnewline

{\small{}$10^3$} & {\small{}$10^5$} & {\small{}220303} & {\small{}*} & {\small{}14929} & \textbf{\small{}3543} & {\small{}937.05} & {\small{}*} & {\small{}138.71} & \textbf{\small{}31.34}\tabularnewline

{\small{}$10^3$} & {\small{}$10^7$} & {\small{}371617} & {\small{}5969} & {\small{}11230} & \textbf{\small{}5121} & {\small{}1791.77} & {\small{}56.92} & {\small{}96.31} & \textbf{\small{}44.50}\tabularnewline

{\small{}$10^3$} & {\small{}$10^8$} & {\small{}1634409} & {\small{}23075} & \textbf{\small{}13465} & {\small{}17371} & {\small{}7250.46} & {\small{}245.84} & \textbf{\small{}133.86} & {\small{}149.43}\tabularnewline

\hline
{\small{}$10^4$} & {\small{}$10^5$} & {\small{}450523} & {\small{}54984} & {\small{}18981} & \textbf{\small{}4756} & {\small{}1908.97} & {\small{}529.99} & {\small{}168.37} & \textbf{\small{}42.61}\tabularnewline

{\small{}$10^4$} & {\small{}$10^6$} & {\small{}248709} & {\small{}*} & {\small{}15876} & \textbf{\small{}6293} & {\small{}1055.40} & {\small{}*} & {\small{}143.12} & \textbf{\small{}54.94}\tabularnewline

{\small{}$10^4$} & {\small{}$10^8$} & {\small{}491118} & {\small{}7959} & {\small{}13184} & \textbf{\small{}7187} & {\small{}2230.07} & {\small{}83.00} & {\small{}125.98} & \textbf{\small{}63.00}\tabularnewline

\end{tabular}
\par\end{centering}
\caption{Iteration counts and runtimes (in seconds) for the Nonconvex QSDP
problem in Subsection~\ref{subsec:nonconvex_sdpsimplex}. The tolerances are set to $10^{-5}$. Entries marked with * did not converge in the time limit of $10800$ seconds. The ATR metric is $0.3831$. \label{tab105:sdpsimplex}}
\end{table}

\begin{table}[!tbh]
\begin{centering}
\begin{tabular}{>{\centering}p{0.7cm}>{\centering}p{0.8cm}|>{\centering}p{1.4cm}>{\centering}p{1.4cm}>{\centering}p{1.2cm}|>{\centering}p{1.4cm}>{\centering}p{1.4cm}>{\centering}p{1.4cm}}
\multicolumn{2}{c|}{\textbf{\scriptsize{}Parameters}} & \multicolumn{3}{c|}{\textbf{\scriptsize{}Iteration Count}} & \multicolumn{3}{c}{\textbf{\scriptsize{}Runtime (seconds)}}\tabularnewline
\hline 
{\normalsize{}$m_f$} & {\normalsize{}$L_{f}$} & {\normalsize{}iALM} & {\normalsize{}RQP} & {\normalsize{}ASL}& {\normalsize{}iALM} & {\normalsize{}RQP} & {\normalsize{}ASL}\tabularnewline
\hline 
\hline 
{\normalsize{}$10^0$} & {\normalsize{}$10^1$} & {\normalsize{}555086} & {\normalsize{}35311} & \textbf{\normalsize{}15699} & {\normalsize{}2257.85} & {\normalsize{}333.79} & \textbf{\normalsize{}138.34}\tabularnewline

{\normalsize{}$10^0$} & {\normalsize{}$10^2$} & {\normalsize{}268608} & {\normalsize{}27247} & \textbf{\normalsize{}2130} & {\normalsize{}1091.32} & {\normalsize{}237.73} & \textbf{\normalsize{}18.24}\tabularnewline

{\normalsize{}$10^0$} & {\normalsize{}$10^3$} & {\normalsize{}355922} & {\normalsize{}26981} & \textbf{\normalsize{}2073} & {\normalsize{}1497.22} & {\normalsize{}243.00} & \textbf{\normalsize{}17.59}\tabularnewline

{\normalsize{}$10^0$} & {\normalsize{}$10^4$} & {\normalsize{}1317510} & {\normalsize{}60908} & \textbf{\normalsize{}2453} & {\normalsize{}5523.22} & {\normalsize{}563.85} & \textbf{\normalsize{}21.21}\tabularnewline

{\normalsize{}$10^0$} & {\normalsize{}$10^5$} & {\normalsize{}*} & {\normalsize{}70646} & \textbf{\normalsize{}10479} & {\normalsize{}*} & {\normalsize{}699.03} & \textbf{\normalsize{}91.90}\tabularnewline

\hline 
{\normalsize{}$10^1$} & {\normalsize{}$10^2$} & {\normalsize{}1297322} & {\normalsize{}68257} & \textbf{\normalsize{}7114} & {\normalsize{}5529.71} & {\normalsize{}676.14} & \textbf{\normalsize{}61.71}\tabularnewline

{\normalsize{}$10^1$} & {\normalsize{}$10^3$} & {\normalsize{}526262} & {\normalsize{}41340} & \textbf{\normalsize{}24596} & {\normalsize{}2254.62} & {\normalsize{}381.20} & \textbf{\normalsize{}212.61}\tabularnewline

{\normalsize{}$10^1$} & {\normalsize{}$10^4$} & {\normalsize{}370204} & {\normalsize{}35879} & \textbf{\normalsize{}2098} & {\normalsize{}1565.84} & {\normalsize{}322.52} & \textbf{\normalsize{}18.06}\tabularnewline

{\normalsize{}$10^1$} & {\normalsize{}$10^5$} & {\normalsize{}998029} & {\normalsize{}42708} & \textbf{\normalsize{}3848} & {\normalsize{}4212.47} & {\normalsize{}387.52} & \textbf{\normalsize{}32.55}\tabularnewline

{\normalsize{}$10^1$} & {\normalsize{}$10^6$} & {\normalsize{}*} & {\normalsize{}36575} & \textbf{\normalsize{}10710} & {\normalsize{}*} & {\normalsize{}325.00} & \textbf{\normalsize{}90.22}\tabularnewline

\hline 
{\normalsize{}$10^2$} & {\normalsize{}$10^4$} & {\normalsize{}689898} & {\normalsize{}39912} & \textbf{\normalsize{}1847} & {\normalsize{}2954.75} & {\normalsize{}377.93} & \textbf{\normalsize{}15.83}\tabularnewline

{\normalsize{}$10^2$} & {\normalsize{}$10^6$} & {\normalsize{}1345701} & {\normalsize{}49506} & \textbf{\normalsize{}3658} & {\normalsize{}5725.54} & {\normalsize{}448.12} & \textbf{\normalsize{}31.57}\tabularnewline

{\normalsize{}$10^2$} & {\normalsize{}$10^7$} & {\normalsize{}*} & {\normalsize{}43571} & \textbf{\normalsize{}12300} & {\normalsize{}*} & {\normalsize{}399.37} & \textbf{\normalsize{}111.21}\tabularnewline

\hline
{\normalsize{}$10^3$} & {\normalsize{}$10^4$} & {\normalsize{}1714445} & {\normalsize{}64949} & \textbf{\normalsize{}1611} & {\normalsize{}7243.63} & {\normalsize{}594.67} & \textbf{\normalsize{}13.94}\tabularnewline

{\normalsize{}$10^3$} & {\normalsize{}$10^5$} & {\normalsize{}596094} & {\normalsize{}40706} & \textbf{\normalsize{}3769} & {\normalsize{}2740.11} & {\normalsize{}363.31} & \textbf{\normalsize{}32.95}\tabularnewline

{\normalsize{}$10^3$} & {\normalsize{}$10^7$} & {\normalsize{}1625487} & {\normalsize{}57454} & \textbf{\normalsize{}7867} & {\normalsize{}6691.03} & {\normalsize{}511.35} & \textbf{\normalsize{}68.70}\tabularnewline

{\normalsize{}$10^3$} & {\normalsize{}$10^8$} & {\normalsize{}*} & {\normalsize{}45759} & \textbf{\normalsize{}18245} & {\normalsize{}*} & {\normalsize{}399.79} & \textbf{\normalsize{}163.00}\tabularnewline

\hline
{\normalsize{}$10^4$} & {\normalsize{}$10^5$} & {\normalsize{}1376159} & {\normalsize{}*} & \textbf{\normalsize{}5030} & {\normalsize{}6145.45} & {\normalsize{}*} & \textbf{\normalsize{}43.15}\tabularnewline

{\normalsize{}$10^4$} & {\normalsize{}$10^6$} & {\normalsize{}995529} & {\normalsize{}51540} & \textbf{\normalsize{}6552} & {\normalsize{}4392.18} & {\normalsize{}489.43} & \textbf{\normalsize{}56.81}\tabularnewline

{\normalsize{}$10^4$} & {\normalsize{}$10^8$} & {\normalsize{}1309587} & {\normalsize{}72323} & \textbf{\normalsize{}8096} & {\normalsize{}5634.84} & {\normalsize{}659.91} & \textbf{\normalsize{}71.30}\tabularnewline
\end{tabular}
\par\end{centering}
\caption{Iteration counts and runtimes (in seconds) for the Nonconvex QSDP
problem in Subsection~\ref{subsec:nonconvex_sdpsimplex}. The tolerances are set to $10^{-6}$. Entries marked with * did not converge in the time limit of $14400$ seconds. The ATR metric is $0.1616$. \label{tab106:qsdpsimplex}}
\end{table}

For our experiments in this subsection, we choose dimensions $(l,n)=(30,100)$. The matrices $A_i$, $B_j$, and $C_i$ are generated so that only 5\% of their entries are nonzero. The entries
of $A_i$, $B_j$, $C_i$, and $d$ (resp. $D$) are generated by sampling
from the uniform distribution ${\cal U}[0,1]$ (resp. ${\cal U}[1,1000]$). We generate the vector $b$ as $b=\cal A(E/n)$, where $E$ is the diagonal matrix in $\Re^{n \times n}$ with all ones on the diagonal. The initial starting point $z_{0}$
is generated as a random matrix in ${\mathbb S}_n^+$. The specific procedure for generating it is described in \cite{AIDAL}. Finally, we choose $(\tau_1, \tau_2)\in\R_{++}^{2}$ so that $L_{f}=\lam_{\max}(\nabla^{2}f)$
and $m_{f}=-\lam_{\min}(\nabla^{2}f)$ are the various values given in the tables of this subsection.

We now describe the specific parameters that ASL, RQP, and iALM choose for this class of problems. Both ASL and RQP choose the initial penalty parameter, $c_1=1$. ASL allows the prox stepsize to be doubled at the end of an iteration if the number of iterations by its ACG call does not exceed $75$. ASL also takes $M^1_0$ defined in its step 1 to be 100 and the initial prox stepsize to be $1/(20m_f)$. Finally, the auxillary parameters of iALM are given by:
\[
B_i=\|A_i\|_F, \quad  L_i=0, \quad \rho_i=0 \quad \forall i \geq 1.
\]

The numerical results are presented in two tables, Table~\ref{tab105:sdpsimplex} and Table~\ref{tab106:qsdpsimplex}. The first table, Table~\ref{tab105:sdpsimplex}, compares ASL with three of the benchmark algorithms namely, iALM, IPL, and RQP. The tolerances are set as $\hat \rho= \hat \eta=10^{-5}$ and a time limit of $10800$ seconds, or 3 hours, is imposed. Table~\ref{tab106:qsdpsimplex} presents the same exact instances as Table~\ref{tab105:sdpsimplex} but now with tolerances set as $\hat \rho=\hat \eta=10^{-6}$ and a time limit of $14400$ seconds, or $4$ hours. Table~\ref{tab106:qsdpsimplex} only compares ASL with iALM and RQP since these were the only two other algorithms to converge for every instance with tolerances set at $10^{-5}$. Entries marked with * did not converge in the time limit.

\subsection{Sparse PCA}\label{subsec:SparsePCA}
We consider the sparse principal components analysis problem studied in \cite{NIPS2014_5615}. That is, given integer $k$, positive scalar pair $(\vartheta,b)\in\r_{++}^{2}$,
and matrix $\Sigma\in S_{+}^{n}$, we consider the following sparse principal component analysis (PCA) problem:
\begin{align*}
\min_{\Pi,\Phi}\  & \left\langle \Sigma,\Pi\right\rangle _{F}+\sum_{i,j=1}^{n}q_{\vartheta}(\Phi_{ij})+\vartheta\sum_{i,j=1}^{n}|\Phi_{ij}|\\
\text{s.t.}\  & \Pi-\Phi=0,\quad(\Pi,\Phi)\in{\cal F}^{k}\times\r^{n\times n}
\end{align*}
where ${\cal F}^{k}=\{M\in S_{+}^{n}:0\preceq M\preceq I,\trc M=k\}$
denotes the $k$--Fantope and $q_{\vartheta}$ is the minimax concave penalty
(MCP) function given by
\[
q_{\vartheta}(t):=\begin{cases}
-t^{2}/(2b), & \text{if }|t|\leq b\vartheta,\\
b\vartheta^{2}/2-\vartheta|t|, & \text{if }|t|>b\vartheta,
\end{cases}\quad\forall t\in\r.
\]

For our experiments in this subsection, we choose $\vartheta=100$ and allow $b$ to vary. Observe that the curvature parameters are $m_f=L_f=1/b$.
We also generate the matrix $\Sigma$ according to an eigenvalue
decomposition $\Sigma=P\Lambda P^{T}$, based on a parameter pair
$(s,k)$, where $k$ is as in the problem description and $s$ is
a positive integer. Specifically, we choose $\Lambda=(100,1,...,1)$,
the first column of $P$ to be a sparse vector whose first $s$ entries
are $1/\sqrt{s}$, and the other entries of $P$ to be sampled randomly
from the standard Gaussian distribution. For our experiments, we fix $s=5$ and allow $k$ to vary. Also, for every problem instance, the initial starting
point is chosen as $(\Pi_{0},\Phi_{0})=(D_{k},0)$ where $D_{k}$ is a diagonal
matrix whose first $k$ entries are 1 and whose remaining entries
are 0. 
\begin{table}[!tbh]
\begin{centering}
\begin{tabular}{>{\centering}p{0.6cm}>{\centering}p{0.75cm}>{\centering}p{0.75cm}|>{\centering}p{1.1cm}>{\centering}p{1.1cm}>{\centering}p{1.1cm}>{\centering}p{1.1cm}|>{\centering}p{1.1cm}>{\centering}p{1.1cm}>{\centering}p{1.1cm}>{\centering}p{1.1cm}}
\multicolumn{3}{c|}{\textbf{\scriptsize{}Parameters}} & \multicolumn{4}{c|}{\textbf{\scriptsize{}Iteration Count}} & \multicolumn{4}{c}{\textbf{\scriptsize{}Runtime (seconds)}}\tabularnewline
\hline 
{\small{}k} & {\small{}$m_f$} & {\small{}$L_{f}$} & {\small{}iALM} & {\small{}IPL} & {\small{}RQP} & {\small{}ASL} & {\small{}iALM} & {\small{}IPL} & {\small{}RQP} & {\small{}ASL}\tabularnewline
\hline 
\hline 
{\small{}5} & {\small{}125}& {\small{}125} & {\small{}*} & {\small{}1438} & {\small{}375618} & \textbf{\small{}428} & {\small{}*} & {\small{}14.84} & {\small{}2776.15} & \textbf{\small{}2.32}\tabularnewline

{\small{}10} & {\small{}125} & {\small{}125} & {\small{}*} & {\small{}1559} & {\small{}40342} & \textbf{\small{}546} & {\small{}*} & {\small{}8.33} & {\small{}216.42} & \textbf{\small{}2.87}\tabularnewline

{\small{}20} & {\small{}125} & {\small{}125} & {\small{}*} & {\small{}1400} & {\small{}13440} & \textbf{\small{}442} & {\small{}*} & {\small{}7.02} & {\small{}67.92} & \textbf{\small{}2.39}\tabularnewline

\hline

{\small{}5} & {\small{}200} & {\small{}200} & {\small{}*} & {\small{}6555} & {\small{}21868} & \textbf{\small{}773} & {\small{}*} & {\small{}43.98} & {\small{}146.79} & \textbf{\small{}4.25}\tabularnewline

{\small{}10} & {\small{}200} & {\small{}200} & {\small{}*} & {\small{}7470} & {\small{}10219} & \textbf{\small{}894} & {\small{}*} & {\small{}47.53} & {\small{}64.34} & \textbf{\small{}5.70}\tabularnewline

{\small{}20} & {\small{}200} & {\small{}200} & {\small{}*} & {\small{}20132} & {\small{}60369} & \textbf{\small{}612} & {\small{}*} & {\small{}118.99} & {\small{}312.28} & \textbf{\small{}4.25}\tabularnewline

\hline

{\small{}5} & {\small{}250} & {\small{}250} & {\small{}*} & {\small{}10391} & {\small{}19365} & \textbf{\small{}622} & {\small{}*} & {\small{}44.25} & {\small{}79.76} & \textbf{\small{}3.96}\tabularnewline

{\small{}10} & {\small{}250} & {\small{}250} & {\small{}211991} & {\small{}32566} & {\small{}143192} & \textbf{\small{}827} & {\small{}574.91} & {\small{}175.81} & {\small{}690.92} & \textbf{\small{}5.10}\tabularnewline

{\small{}20} & {\small{}250} & {\small{}250} & {\small{}236490} & {\small{}199353} & {\small{}*} & \textbf{\small{}558} & {\small{}628.55} & {\small{}985.49} & {\small{}*} & \textbf{\small{}3.22}\tabularnewline
\hline

{\small{}5} & {\small{}1000} & {\small{}1000} & {\small{}*} & {\small{}567358} & {\small{}63633} & \textbf{\small{}1818} & {\small{}*} & {\small{}2873.66} & {\small{}312.34} & \textbf{\small{}11.97}\tabularnewline

{\small{}10} & {\small{}1000} & {\small{}1000} & {\small{}*} & {\small{}*} & {\small{}*} & \textbf{\small{}932} & {\small{}*} & {\small{}*} & {\small{}*} & \textbf{\small{}14.59}\tabularnewline

{\small{}20}  & {\small{}1000} & {\small{}1000} & {\small{}*} & {\small{}*} & {\small{}*} & \textbf{\small{}1581} & {\small{}*} & {\small{}*} & {\small{}*} & \textbf{\small{}12.62}\tabularnewline

\end{tabular}
\par\end{centering}
\caption{Iteration counts and runtimes (in seconds) for the Sparse PCA
problem in Subsection~\ref{subsec:SparsePCA}. The tolerances are set to $10^{-5}$. Entries marked with * did not converge in the time limit of $3600$ seconds. The ATR metric is $0.0306$. \label{tab105:spca}}
\end{table}

We now describe the specific parameters that ASL, RQP, and iALM choose for this class of problems. Both ASL and RQP choose the initial penalty parameter, $c_1=1$. ASL allows the prox stepsize to be doubled at the end of an iteration if the number of iterations by its ACG call does not exceed $4$. ASL also takes $M^1_0$ defined in its step 1 to be 1 and the initial prox stepsize to be $1/(2m_f)$. Finally, the auxillary parameters of iALM are chosen as:
\[
B_i=1, \quad  L_i=0, \quad \rho_i=0 \quad \forall i \geq 1,
\]
based off the relaxed, but unverified assumption that its iterates lie in ${\cal F}^{k} \times {\cal F}^{k}$.

The numerical results are presented in Table~\ref{tab105:spca}. Table~\ref{tab105:spca} compares ASL with three of the benchmark algorithms namely, iALM, IPL, and RQP. The tolerances are set as $\hat \rho= \hat \eta=10^{-5}$ and a time limit of $3600$ seconds, or 1 hour, is imposed. Entries marked with * did not converge in the time limit.

\subsection{Bounded Matrix Completion (BMC)}\label{bounded matrix completion}
We consider the bounded matrix completion problem studied in \cite{yao2017efficient}. That is, given a dimension pair $(p,q)\in\n^{2}$, positive scalar triple $(\upsilon,\tau_m,\theta)\in\r_{++}^{3}$,
scalar pair $(u,l)\in\r^{2}$, matrix $Q\in\r^{p\times q}$, and indices
$\Omega$, we consider the following bounded matrix completion (BMC) problem: 
\begin{align*}
\min_{X}\  & \frac{1}{2}\|P_{\Omega}(X-Q)\|^{2}+\tau_m\sum_{i=1}^{\min\{p,q\}}\left[\kappa(\sigma_{i}(X))-\kappa_{0}\sigma_{i}(X)\right]+{\tau_m\kappa_0}\|X\|_{*}\\
\text{s.t.}\  & l\leq X_{ij}\leq u\quad\forall(i,j)\in\{1,...,p\}\times\{1,...,q\},
\end{align*}
where $\|\cdot\|_{*}$ denotes the nuclear norm, the function $P_{\Omega}$
is the linear operator that zeros out any entry not in $\Omega$,
the function $\sigma_{i}(X)$ denotes the $i^{\rm th}$ largest singular
value of $X$, and 
\[
\kappa_{0}:=\frac{\upsilon}{\theta},\quad\kappa(t):=\upsilon\log\left(1+\frac{|t|}{\theta}\right)\quad\forall t\in\r.
\]
\begin{table}[!tbh]
\begin{centering}
\begin{tabular}{>{\centering}p{0.6cm}>{\centering}p{0.75cm}>{\centering}p{0.75cm}>{\centering}p{0.75cm}|>{\centering}p{1.0cm}>{\centering}p{1.0cm}|>{\centering}p{1.4cm}>{\centering}p{1.4cm}}
\multicolumn{4}{c|}{\textbf{\scriptsize{}Parameters}} & \multicolumn{2}{c|}{\textbf{\scriptsize{}Iteration Count}} & \multicolumn{2}{c}{\textbf{\scriptsize{}Runtime (seconds)}}\tabularnewline
\hline 
{\small{}$\theta$} & {\small{}$\tau_m$} & {\small{}$m_{f}$}& {\small{}$L_{f}$} & {\small{}RQP} & {\small{}ASL} & {\small{}RQP} & {\small{}ASL}\tabularnewline
\hline 
\hline 
{\small{}1/2} & {\small{}0.5}& {\small{}2} & {\small{}2} & \textbf{\small{}}130 & \textbf{\small{}79} & {\small{}209.07}& \textbf{\small{}139.52}\tabularnewline

{\small{}1/2} & {\small{}1} & {\small{}4} & {\small{}4} & {\small{}128} & \textbf{\small{}119} & {\small{}207.21} &  \textbf{\small{}189.75}\tabularnewline

{\small{}1/2} & {\small{}2} & {\small{}8} & {\small{}8} & {\small{}1233} & \textbf{\small{}457} &{\small{}2075.16}& \textbf{\small{}931.91}\tabularnewline

\hline

{\small{}1/3} & {\small{}0.5} & {\small{}4.5} & {\small{}4.5} & {\small{}384} & \textbf{\small{}51} & {\small{}1229.60} & \textbf{\small{}61.22}\tabularnewline

{\small{}1/3} & {\small{}1} & {\small{}9} & {\small{}9} & {\small{}513} & \textbf{\small{}76} & {\small{}1360.69} & \textbf{\small{}101.04}\tabularnewline

{\small{}1/3} & {\small{}2} & {\small{}18} & {\small{}18} & {\small{}*} & \textbf{\small{}494} & {\small{}*} & \textbf{\small{}1001.82}\tabularnewline

\hline

{\small{}1/4} & {\small{}0.5} & {\small{}8} & {\small{}8} & {\small{}601} & \textbf{\small{}66} & {\small{}928.01} & \textbf{\small{}85.28}\tabularnewline

{\small{}1/4} & {\small{}1} & {\small{}16} & {\small{}16} & {\small{}680} & \textbf{\small{}90} & {\small{}1077.89} &  \textbf{\small{}147.76}\tabularnewline
\hline 

{\small{}1/5} & {\small{}0.5} & {\small{}12.5} & {\small{}12.5} & {\small{}488} & \textbf{\small{}193} & {\small{}1653.75} & \textbf{\small{}313.51}\tabularnewline

{\small{}1/5} & {\small{}1} & {\small{}25} & {\small{}25} & {\small{}859} & \textbf{\small{}227} & {\small{}1494.94} & \textbf{\small{}475.09}\tabularnewline
\hline 

{\small{}1/6} & {\small{}0.5} & {\small{}18} & {\small{}18} & {\small{}838} & \textbf{\small{}96} & {\small{}1359.45} &  \textbf{\small{}137.72}\tabularnewline

{\small{}1/6}  & {\small{}1} & {\small{}36} & {\small{}36} & {\small{}617} & \textbf{\small{}221} & {\small{}962.52} & \textbf{\small{}358.25}\tabularnewline
\hline 

{\small{}1/7} & {\small{}0.5} & {\small{}24.5} & {\small{}24.5} & {\small{}770} & \textbf{\small{}142} & {\small{}1232.90} & \textbf{\small{}195.66}\tabularnewline

{\small{}1/7}  & {\small{}1} & {\small{}49} & {\small{}49} & {\small{}789} & \textbf{\small{}355} & {\small{}1213.75} & \textbf{\small{}580.04}\tabularnewline

\end{tabular}
\par\end{centering}
\caption{Iteration counts and runtimes (in seconds) for the BMC
problem in Subsection~\ref{bounded matrix completion}. The tolerances are set to $10^{-3}$. Entries marked with * did not converge in the time limit of $7200$ seconds. The ATR metric is $0.3079$. \label{tab103:bmc}}
\end{table}

We first describe the parameters considered for the above problem and some of its properties. First, the matrix $Q$ is the user-movie ratings data matrix of
the MovieLens 100K dataset\footnote{See the MovieLens 100K dataset containing 610 users and 9724 movies which can be found in https://grouplens.org/datasets/movielens/}. Second, $\upsilon$ is chosen to be $0.5$ and $\tau_m$ and $\theta$ are allowed to vary. Third, the curvature parameters are $m_f=2\upsilon\tau_m/\theta^{2}$ and $L_f=\max\left\{1,m_f\right\}$. Fourth, the bounds are set to $(l,u)=(0,5)$ and the initial starting point is chosen as $X_0=0$. Finally, the above optimization problem can be written in the form:
\begin{align*}
\min_{X}\  & f(X)+h(X)\\
\text{s.t.}\  & {\cal A}(X)\in S,
\end{align*}
where 
\begin{gather*}
f(X)=\frac{1}{2}\|P_{\Omega}(X-Q)\|^{2}+\tau_m\sum_{i=1}^{\min\{p,q\}}\left[\kappa(\sigma_{i}(X))-\kappa_{0}\sigma_{i}(X)\right],\quad h(X)={\tau_m\kappa_0}\|X\|_{*},\\
\cal A(X)=X,\quad S=\left\{ Z\in\r^{p\times q}:l\leq Z_{ij}\leq u,\:(i,j)\in\{1,...,p\}\times\{1,...,q\}\right\}.\\
\end{gather*}

To deal with the more generalized constraints $\cal A(X) \in S$, our method, ASL, considers the following augmented Lagrangian function and Lagrange multiplier update:
\begin{align*}
&\mathcal L_c(z,p):=f(z)+h(z)-\frac{\|p\|^2}{2c}+\frac{c}{2}\|(Az+\frac{p}{c})-\Pi_S(Az+\frac{p}{c})\|^2;\\
&p_k:=p_{k-1}+c_{k}(Az_{k}-\Pi_S(Az_k+\frac{p_{k-1}}{c_k})),
\end{align*}
where $\Pi_S$ denotes the projection onto the set S. We only compare ASL to the RQP method since RQP was the only other method developed and coded to deal with these generalized constraints. 

We now describe the specific parameters that ASL and RQP choose for this class of problems. Both ASL and RQP choose the initial penalty parameter, $c_1=500$. ASL allows the prox stepsize to be doubled at the end of an iteration if the number of iterations by its ACG call does not exceed $4$. Finally, ASL also takes $M^1_0$ defined in its step 1 to be 1 and the initial prox stepsize to be $10/(m_f)$.

The numerical results are presented in Table~\ref{tab103:bmc}. Table~\ref{tab103:bmc} compares ASL with RQP. The tolerances are set as $\hat \rho= \hat \eta=10^{-3}$ and a time limit of $7200$ seconds, or 2 hours, is imposed. Entries marked with * did not converge in the time limit. 
\subsection{Comments about the numerical results}
Overall, the two adaptive methods ASL and RQP were the most reliable and consistent, converging in almost every instance. ASL was clearly the most efficient, often converging much faster than RQP particularly when the
required accuracy was high. As demonstrated by the results in Tables \ref{tab104:qpsimplex} and \ref{tab106:qpsimplex}, and the ones in  Tables \ref{tab105:sdpsimplex} and \ref{tab106:qsdpsimplex}, the ATR metric improves (decreases) as the required accuracy increases. Finally, ASL
worked extremely fast on the problem classes of Subsections \ref{subsec:nonconvex_qpbox} and \ref{subsec:SparsePCA}
as demonstrated by the results in Tables \ref{tab105:qpbox} and \ref{tab105:spca},
respectively.
\footnotesize
\begin{appendices}
\section{ADAP-FISTA algorithm} \label{sec:acg}
\subsection{ADAP-FISTA method} 
This subsection presents an adaptive ACG variant, called ADAP-FISTA, which is an important tool in the development of the AS-PAL method. We first introduce the assumptions on the problem  it solves. 
ADAP-FISTA considers the following problem
\begin{equation}\label{OptProb1}
\min\{ \psi(x):= \psi_s(x) + \psi_n(x) : x \in \R^n\}
\end{equation}
where $\psi_s$ and $\psi_n$ are assumed to satisfy the following assumptions:

\begin{itemize}
\item[\textbf{(I)}] $\psi_n:\R^n \to \R\cup\{+\infty\}$ is a possibly nonsmooth convex function;

\item[\textbf{(II)}]
$\psi_s: \R^n\to \R$ is a differentiable function and
there exists $\bar L \geq 0$ such that
\begin{equation}\label{ineq:uppercurvature1}
\|\nabla \psi_s(z') -  \nabla \psi_s(z)\|\le \bar L \|z'-z\| \quad \forall z,z' \in \R^n.
\end{equation}
\end{itemize}

We now describe the type of approximate solution that ADAP-FISTA aims to find. 

\vgap

\noindent{\bf Problem A:} Given $\psi$ satisfying the above assumptions,
a point $x_0 \in \dom \psi_n$, a parameter $\sigma\in(0,\infty)$,
the problem is to find a pair $(y,u) \in \dom \psi_n \times \R^n$ such that
\begin{gather}
    \|u\|\leq \sigma \|y-x_0\|, \quad u \in \nabla \psi_s(y)+\partial \psi_n(y).\label{acg problem}
\end{gather}

We are now ready to present the ADAP-FISTA algorithm below. 

\noindent\begin{minipage}[t]{1\columnwidth}%
\rule[0.5ex]{1\columnwidth}{1pt}

\noindent \textbf{ADAP-FISTA Method}

\noindent \rule[0.5ex]{1\columnwidth}{1pt}%
\end{minipage}

\begin{itemize}
\item[{\bf 0.}] Let initial point $ x_0 \in \dom \psi_n$ and scalars $\mu >0$, $L_0>\mu$, $\chi \in (0,1)$, $\beta>1$, and $\sigma>0$ be given, and set $ y_0=x_0 $, $ A_0=0 $, $\tau_0=1$, and $ j=0$;

\item[{\bf 1.}] Set $L_{j+1}=L_j$;
\item[{\bf 2.}]	Compute
		\begin{equation}\label{def:ak-sfista1}
		a_j=\frac{\tau_{j}+\sqrt{\tau_{j}^2+4\tau_{j} A_{j}(L_{j+1}-\mu)}}{2(L_{j+1}-\mu)}, \quad \tx_{j}=\frac{A_{j}y_{j}+a_j x_j}{A_j+a_j},
		\end{equation}
		\begin{equation}
		y_{j+1}:=\underset{u\in \dom \psi_n}\argmin\left\lbrace q_j (u;\tx_{j},L_{j+1}) 
		:= \ell_{\psi_s}(u;\tilde x_{j}) + \psi_n(u) + \frac{L_{j+1}}{2}\|u-\tx_{j}\|^2\right\rbrace,
		\label{eq:ynext-sfista1}
		\end{equation}
		If  the inequality
		\begin{equation}\label{ineq check}
		\ell_{\psi_s}(y_{j+1};\tilde x_{j})+\frac{(1-\chi) L_{j+1}}{2}\|y_{j+1}-\tilde x_{j}\|^2\geq \psi_s(y_{j+1})
		\end{equation}
		holds go to step~3; else set $L_{j+1} \leftarrow \beta L_{j+1} $ and repeat step~2;
\item[{\bf 3.}] Compute
\begin{align}
A_{j+1}&=A_j+a_j, \quad \tau_{j+1}= \tau_j + a_j\mu,  \label{eq:taunext-sfista1} \\
s_{j+1}&=(L_{j+1}-\mu)(\tilde x_j-y_{j+1}),\label{eq:sk}\\
\quad x_{j+1}&= \frac{1}{\tau_{j+1}} \left[\mu a_j y_{j+1} + \tau_j x_j-a_js_{j+1} \right];\label{eq:xnext-sfista1}
\end{align}
	
\item[{\bf 4.}]  If the inequality
\begin{align}
&\|y_{j+1}-x_{0}\|^{2} \geq \chi A_{j+1}L_{j+1} \|y_{j+1}-\tilde x_{j}\|^2, \label{ineq: ineq 1}
\end{align}
holds, then  go to step 5; otherwise, stop with {\bf failure};

\item[{\bf 5.}] Compute 
\begin{equation}\label{def:uk}
u_{j+1}=\nabla \psi_s(y_{j+1})-\nabla \psi_s(\tilde x_{j})+L_{j+1}(\tilde x_j-y_{j+1}).   
\end{equation} If the inequality
		\begin{equation}\label{u sigma criteria}
		\|u_{j+1}\| \leq \sigma \|y_{j+1}-x_0\|
		\end{equation}
		holds then stop with {\bf success} and output $(y,u):=(y_{j+1},u_{j+1})$; otherwise,
		 $ j \leftarrow j+1 $ and go to step~1.
\end{itemize}
\noindent \rule[0.5ex]{1\columnwidth}{1pt}




We now make some remarks about ADAP-FISTA. First, usual FISTA methods for solving the strongly convex version of \eqref{OptProb1} consist of repeatedly invoking only steps 2 and 3 of ADAP-FISTA either
with a static Lipschitz constant,
namely, $L_{j+1}=L$ for all $j \ge0$ for some $L\geq \bar L$, or by adaptively searching for a suitable Lipschitz $L_{j+1}$ (as in step 2 of ADAP-FISTA) satisfying
a condition similar to \eqref{ineq check}.
Second, the pair $(y_{j+1},u_{j+1})$ always satisfies the inclusion in \eqref{acg problem} (see Lemma~\ref{lem:gamma-sfista0} below) so if ADAP-FISTA stops successfully in step 5, or equivalently \eqref{u sigma criteria} holds, the pair solves Problem A above. Finally, if condition \eqref{ineq: ineq 1} in step 4 is never violated, ADAP-FISTA must stop successfully in step 5 (see Proposition~\ref{prop:nest_complex1} below).

We now discuss how ADAP-FISTA compares with existing ACG variants for solving \eqref{OptProb1} under the assumption that $\psi_s$ is $\mu$-strongly convex.
Under this assumption,
FISTA variants have
been studied, for example, in \cite{CurvatureFree, florea2018accelerated, nesterov1983method, fistaReport2021, MontSvaiter_fista}, while
other ACG variants have been studied, for example, in \cite{nesterov2012gradient, YHe1, YHe2}.
A crucial difference
between ADAP-FISTA and these
variants is that:
i) ADAP-FISTA stops based on a different relative criterion, namely, \eqref{u sigma criteria} (see Problem A above) and
attempts to approximately
solve
\eqref{OptProb1}
in this sense
even when
$\psi_s$ is not $\mu$-strongly convex, and
ii)
ADAP-FISTA provides a key
and easy to check inequality whose validity at every iteration guarantees its successful termination.
On the other hand, ADAP-FISTA
shares similar features with these other methods in that:
i) it has a reasonable iteration complexity guarantee
regardless of whether
it succeeds or fails,
and
ii) it successfully
terminates when
$\psi_s$ is $\mu$-strongly convex (see Propositions~\ref{prop:nest_complex1}-\ref{strongly convex fista} below). Moreover, like the method in \cite{florea2018accelerated}, ADAP-FISTA
adaptively searches for a suitable Lipschitz estimate
$L_{j+1}$ that is used in \eqref{eq:ynext-sfista1}.

We now present the main convergence results of ADAP-FISTA, which is invoked by AS-PAL for solving the sequence of subproblems \eqref{eq:approx_primal_update}. The first result, namely
Proposition~\ref{prop:nest_complex1} below, gives an
iteration complexity bound regardless if ADAP-FISTA terminates with success or failure and shows that if ADAP-FISTA successfully stops, then it obtains a stationary solution of \eqref{OptProb1} with respect to a relative error criterion. The second result, namely Proposition~\ref{strongly convex fista} below, shows that ADAP-FISTA always stops successfully
whenever $\psi_s$ is $\mu$-strongly convex.

\begin{proposition}\label{prop:nest_complex1} 
The following statements about ADAP-FISTA hold:
\begin{itemize}
\item[(a)] if
$L_0 = {\cal O}(\bar L)$, it
always stops (with either success or failure) in at most 
\[\cal O_1\left(\sqrt{\frac{\bar L}{\mu}}\log^+_0 (\bar L) \right)\] 
iterations/resolvent evaluations;
\item[(b)] if it stops successfully, it terminates with a pair
$(y,u) \in \dom \psi_{n} \times \Re^{n}$ satisfying 
\begin{align}
&u \in \nabla \psi_s(y)+\partial \psi_n(y) \label{u inclusion 1};\\
&\|u\|\leq \sigma\|y-x_0\|. \label{u termination 1}
\end{align}
\end{itemize}
\end{proposition}
\begin{proposition}\label{strongly convex fista}
If $\psi_s$ is $\mu$-convex, then ADAP-FISTA always terminates with success and
its output
$(y,u)$, in addition
to satisfying
\eqref{u inclusion 1} and \eqref{u termination 1},
also  satisfies
the inclusion $u \in \partial (\psi_s+\psi_n)(y)$.
\end{proposition}
The rest of this section is broken up into two subsections which are dedicated to proving Proposition~\ref{prop:nest_complex1} and Proposition~\ref{strongly convex fista}, respectively.
\subsection{Proof of Proposition~\ref{prop:nest_complex1}}  
This subsection is dedicated to proving Proposition~\ref{prop:nest_complex1}. The first lemma below presents key definitions and inequalities used in the convergence analysis of ADAP-FISTA.

\begin{lemma}\label{lem:gamma-sfista0}
Define
\begin{equation}\label{omega zeta}
\omega=\beta/(1-\chi), \quad \zeta:=\bar L+\max\{L_0,\omega \bar L\}.
\end{equation}
Then, the following statements hold:
\begin{itemize}
\item[(a)] $\{L_j\}$ is nondecreasing;
\item[(b)] for every $j \ge 0$, we have
\begin{align}
&\tau_j = 1+A_j\mu, \quad \frac{\tau_j  A_{j+1}}{ a_j^2}=L_{j+1}-\mu;\label{tauproperty}\\[0.5em]
&L_0\leq L_j\leq \max\{L_0,\omega \bar L\}\label{upper bound};\\[0.8em]
& u_{j+1} \in \nabla \psi_s(y_{j+1}) + \partial \psi_n(y_{j+1}), \quad \|u_{j+1}\|\leq \zeta \|y_{j+1}-\tilde x_{j}\|\label{ubound-1}.
\end{align}
\end{itemize}
\end{lemma}
\begin{proof} (a) It is clear from the update rule in the beginning of Step 1 that $\{L_j\}$ is nondecreasing.

(b) The first equality in \eqref{tauproperty} follows directly from both of the relations in \eqref{eq:taunext-sfista1}. The second equality in \eqref{tauproperty} follows immediately from the definition of $a_j$ in \eqref{def:ak-sfista1} and the first relation in  \eqref{eq:taunext-sfista1}. 

We prove \eqref{upper bound} by induction. It clearly holds for $j=0$. Suppose now $\eqref{upper bound}$ holds for
$j \ge 0$ and let us show that it holds for $j+1$. Note
that if $L_{j+1}=L_j$, then relation
\eqref{upper bound} immediately holds.
Assume then that
$L_{j+1}>L_j$. It then follows from
the way $L_{j+1}$ is chosen in step 1 that \eqref{ineq check} is not satisfied with $L_{j+1}/\beta$. This fact together with the inequality \eqref{ineq:uppercurvature1} at the points $(y_{j+1},\tilde x_j)$ imply that 
\begin{equation}\label{Lipschitz ineq}
    \ell_{\psi_s}(y_{j+1};\tilde x_j)+\frac{(1-\chi) L_{j+1}}{2\beta}\|y_{j+1}-\tilde x_j\|^2 < \psi_s(y_{j+1}) \overset{\eqref{ineq:uppercurvature1}}{\le}
    \ell_{\psi_s}(y_{j+1};\tilde x_j)+\frac{\bar L}{2}\|y_{j+1}-\tilde x_j\|^2.
\end{equation}
The relation in \eqref{upper bound} then immediately follows from the definition of $\omega$ in \eqref{omega zeta}.

Now, by the definition of $u_{j+1}$ in \eqref{def:uk}, triangle inequality, \eqref{ineq:uppercurvature1}, the bound \eqref{upper bound} on $L_{j+1}$, and the definition of $\zeta$ we have
\[\frac{\|u_{j+1}\|}{\|y_{j+1}- \tx_{j}\| } \overset{\eqref{def:uk}}{\le} 
\frac{ \| \nabla \psi_s(y_{j+1}) - \nabla \psi_s(\tx_{j}) \| }{\|y_{j+1}- \tx_{j}\| }+ L_{j+1} \overset{\eqref{ineq:uppercurvature1}}{\le} \bar L+L_{j+1}
\overset{\eqref{upper bound}}{\le} \zeta\]
which immediately implies the inequality in \eqref{ubound-1}.
It follows from \eqref{eq:ynext-sfista1} and its associated optimality condition that
$0 \in \nabla \psi_s(\tx_{j}) + \partial \psi_n(y_{j+1})-L_{j+1}(\tilde x_j-y_{j+1})$, which in view of the definition of $u_{j+1}$ in \eqref{def:uk} implies the inclusion in \eqref{ubound-1}. 
\end{proof}

The result below gives some estimates on the sequence $\{A_j\}$, which will be important for the convergence analysis of the method.
\begin{lemma}\label{lm:Ak-est-fista}
Define 
\begin{equation}\label{Q def}
Q:= 2 \sqrt{ \frac{\max\{L_0,\omega {\bar L}\}}{\mu}}
\end{equation}
where $\omega$ is as in \eqref{omega zeta}.
Then, for every $j \ge 1$, we have
\beq
A_jL_j\ge\label{eq:Akest-sfista1}
\max \left\{ 
\frac{ j^2}{4} ,  \left(1+Q^{-1}\right)^{2(j-1)}
 \right\}.
\eeq
\end{lemma}
\begin{proof}
Let integer $j \ge 1$ be given. Define $\xi_{j}=1/(L_{j}-\mu)$.
Using the first equality in \eqref{eq:taunext-sfista1} and the definition of $a_j$ in \eqref{def:ak-sfista1},
we have that for every $i \le j$,
\[
A_{i} \overset{\eqref{eq:taunext-sfista1}}{=} A_{i-1}+ a_{i-1} \overset{\eqref{def:ak-sfista1}}{\ge} A_{i-1} +  \left( \frac{\tau_{i-1} \xi_{i}}2 + \sqrt{\tau_{i-1} \xi_{i} A_{i-1}} \right) \ge 
\left(  \sqrt{A_{i-1}} + \frac12 \sqrt{\tau_{i-1} \xi_{i}} \right)^2.
\]
Passing the above inequality
to its square root and
using Lemma~\ref{lem:gamma-sfista0}(a) and the
fact that 
\eqref{tauproperty} implies that $\tau_{i-1} \ge \max\{1,\mu A_{i-1}\}$, we then conclude that for every $i \le j$,

\begin{align}
\sqrt{A_{i}} - \sqrt{A_{i-1}} &\ge \frac12 \sqrt{ \xi_{i}} \ge \frac12 \sqrt{ \xi_{j} }\label{first inequality}\\
\sqrt{\frac{A_{i}}{A_{i-1}}} &\ge   1 + \frac12 \sqrt{\mu \xi_{i}}
 \ge  1 + \frac12 \sqrt{\mu \xi_{j}}\ge 1 + Q^{-1} \label{second inequality}
\end{align}
where the last inequality in \eqref{second inequality} follows from the definition of $\xi_j$, the relation in \eqref{upper bound}, and the definition of $Q$ in \eqref{Q def}.
Adding the inequality in \eqref{first inequality} from $i=1$ to $i=j$ and
using the fact that $A_0=0$, we conclude that $
\sqrt{A_j} \ge j \sqrt{\xi_j} /2$
and hence that the
first bound in \eqref{eq:Akest-sfista1} holds in view of the fact that $\xi_j \geq 1/L_j$.
Now, multiplying the inequality in \eqref{second inequality} from $i=2$ to $i=j$ and
using Lemma~\ref{lem:gamma-sfista0}(a) and
the fact that $A_1= \xi_1$, we conclude that $
\sqrt{A_j} \ge \sqrt{\xi_1} (1+Q^{-1})^{j-1} \ge \sqrt{\xi_j} (1+Q^{-1})^{j-1}$,
and hence that the
second bound in \eqref{eq:Akest-sfista1} holds in view of the fact that $\xi_j \geq 1/L_j$.
\end{proof}
\begin{proposition}\label{exact ACG complexity}
Let $\zeta$ and $Q$ be as in \eqref{omega zeta} and \eqref{Q def}, respectively. ADAP-FISTA
always stops (with either success or failure)
and does so by performing
at most 
\begin{equation}\label{eq:eq1}
\left\lceil\left(1+Q\right)\log^+_0\left(\frac{\zeta^2}{\chi \sigma^{2}}\right)+1\right\rceil+\left\lceil \frac{\log^+_0(\bar L/((1-\chi)L_0))}{\log \beta}\right\rceil
\end{equation}
iterations/resolvent evaluations.
\end{proposition}
\begin{proof}
Let $l$ denote the first quantity in \eqref{eq:eq1}. Using 
this definition and the
inequality $\log (1+ \alpha) \ge \alpha/(1+\alpha)$ for
any $\alpha>-1$, it is easy to verify that
\begin{equation}\label{Q bound}
\left(1+Q^{-1}\right)^{2(l-1)} \ge 
\frac{\zeta^2}{\chi \sigma^2}.
\end{equation}
We claim that ADAP-FISTA terminates with success or failure in at most $l$ iterations.
Indeed, it suffices to show that
if ADAP-FISTA has not stopped with failure
up to (and including) the
$l$-th iteration, then
it must stop successfully
at the $l$-th iteration.
So, assume that
ADAP-FISTA has not stopped with failure
up to the
$l$-th iteration.
In view of step~4 of ADAP-FISTA,
it follows that
\eqref{ineq: ineq 1} holds with $j=l-1$.

This observation together with  
the inequality in \eqref{ubound-1} with $j=l-1$, \eqref{eq:Akest-sfista1} with $j=l$,
and \eqref{Q bound},
then imply that
\begin{equation}\label{bound diff-1}
\|y_{l}-x_{0}\|^{2} \overset{\eqref{ineq: ineq 1}}{\geq} \chi A_{l}L_{l} \|y_{l}-\tilde x_{l-1}\|^2
\overset{\eqref{ubound-1}}{\ge}
\frac{\chi}{\zeta^2}A_lL_l \|u_{l}\|^2\overset{\eqref{eq:Akest-sfista1}}{\geq}\frac{\chi}{\zeta^2}
\left(1+Q^{-1} \right)^{2(l-1)}
\|u_l\|^2 \overset{\eqref{Q bound}}{\geq} \frac{1}{\sigma^2}\|u_l\|^2,
\end{equation}
and hence that
\eqref{u sigma criteria} is satisfied.
In view of Step 5 of
ADAP-FISTA, the method
must
successfully stop at the
end of the $l$-th iteration.
We have thus shown that
the above claim holds.
Moreover, in view of \eqref{upper bound},
it follows that
the second term in 
\eqref{eq:eq1} is a bound
on the total number of times $L_j$ is multiplied by $\beta$ and step 2 is repeated. Since exactly one resolvent evaluation occurs every time step 2 is executed, the desired conclusion follows.
\end{proof}

We are now ready to give the proof of Proposition~\ref{prop:nest_complex1}.
\begin{proof}[of Proposition~\ref{prop:nest_complex1}]
(a) The result immediately follows from Proposition~\ref{exact ACG complexity} and the assumption that $L_0 = {\cal O}(\bar L)$.

(b) This is immediate from the termination criterion \eqref{u sigma criteria} in step 5 of ADAP-FISTA and the inclusion in \eqref{ubound-1}.

\end{proof}

\subsection{Proof of Proposition~\ref{strongly convex fista}}
This subsection is dedicated to proving Proposition~\ref{strongly convex fista}. Thus, for the remainder of this subsection, assume that $\psi_s$ is $\mu$-strongly convex. The first lemma below presents important properties of the iterates generated by ADAP-FISTA.
\begin{lemma}\label{lem:gamma-sfista}
For every $ j\ge 0 $ and $x \in \R^n$, define 
	\begin{align}
		\gamma_j(x)&:=\ell_{\psi_s}(y_{j+1},\tilde x_j) + \psi_n(y_{j+1})+\inner{s_{j+1}}{x - y_{j+1}}+\frac{\mu}{2}\|y_{j+1}-\tilde x_j\|^2
		+ \frac{\mu}2 \|x-y_{j+1}\|^2, \label{def:gamma-sfista}
	\end{align}
where $\psi:=\psi_s+\psi_n$ and $s_{j+1}$ are as in \eqref{OptProb1} and \eqref{eq:sk}, respectively.
Then, for every $j \ge 0$, we have:
	\begin{align}
		y_{j+1} &= \argmin _{x}\left\{\gamma_{j}(x)+\frac{L_{j+1}-\mu}{2}\left\|x-\tilde{x}_{j}\right\|^{2}\right\};\label{eq:min-sfista}\\
		x_{j+1}&=\underset{x \in \R^n}\argmin\left\{a_{j} \gamma_{j}(x)+\tau_j \left\|x-x_{j}\right\|^{2} /2 \right\}.\label{xgamma}
		\end{align}
\end{lemma}
\begin{proof}
Since $\nabla \gamma_j(y_{j+1})=s_{j+1}$, it follows from
\eqref{eq:sk} that $y_{j+1}$ satisfies the optimality condition for \eqref{eq:min-sfista},
and thus the relation in \eqref{eq:min-sfista} follows. Furthermore, we have that:
\begin{align*}
 a_j \nabla \gamma_j(x_{j+1})+\tau_j(x_{j+1}-x_j)&=a_js_{j+1}+a_j\mu(x_{j+1}-y_{j+1})+\tau_j(x_{j+1}-x_{j})\\
 &\overset{\eqref{eq:taunext-sfista1}}{=} a_j s_{j+1} -\mu a_jy_{j+1}-\tau_jx_j+ \tau_{j+1}x_{j+1}
 \overset{\eqref{eq:xnext-sfista1}}{=}0
\end{align*}
and thus \eqref{xgamma} follows. 
\end{proof}

Before stating the next lemma, recall that if a closed function $\Psi:\R^n\to\R\cup\{+\infty\}$ is $\nu$-convex with modulus $\nu>0$, then it has an unique global minimum $z^*$ and
\begin{equation}\label{ineq:nu-convex}
\Psi(z^*) +\frac{\nu}{2}\|\cdot - z^*\|^2\leq \Psi(\cdot). 
\end{equation}

\begin{lemma} \label{lm:hysub-3-sfista}
	For every $ j\ge 0 $ and $x \in \R^n$, we have
	\begin{align}\label{ineq:recur}
	A_j\gamma_j(y_j) &+ a_j\gamma_j(x) + \frac{\tau_j}2 \|x_j - x\|^2 - \frac{\tau_{j+1}}2 \|x_{j+1} - x\|^2 \nonumber \\
	&\ge A_{j+1}\psi(y_{j+1})+\frac{\chi A_{j+1} L_{j+1}}{2}\|y_{j+1}-\tilde x_{j}\|^2.  
	\end{align}
\end{lemma}
\begin{proof}
Using \eqref{xgamma}, the second identity in \eqref{eq:taunext-sfista1}, and the fact that $\Psi_j:=a_j\gamma_j(\cdot)+\tau_j \|\cdot-x_j\|^2/2$ is  $(\tau_j+\mu a_j)$-convex, it follows from 
\eqref{ineq:nu-convex} with $\Psi=\Psi_j$ and $\nu=\tau_{j+1}$ that 
\begin{align*}
a_j\gamma_j(x) + \frac{\tau_j}2 \|x-x_j\|^2 - \frac{\tau_{j+1}}2 \|x-x_{j+1}\|^2 \ge a_j\gamma_j(x_{j+1}) + \frac{\tau_j}2 \|x_{j+1}-x_j\|^2 \quad
\forall x \in \R^n.
\end{align*}
	Using the convexity of $ \gamma_j $, the definitions of $A_{j+1}$ and  $ \tx_j $ in \eqref{eq:taunext-sfista1} and \eqref{def:ak-sfista1}, respectively, and the second equality in \eqref{tauproperty}, we have
	\begin{align*}
		A_j\gamma_j(y_j) &+ a_j\gamma_j(x_{j+1}) + \frac{\tau_j}2 \|x_{j+1}-x_j\|^2 \\
		&\ge A_{j+1} \gamma_j\left( \frac{A_jy_j+a_jx_{j+1}}{A_{j+1}} \right) + \frac{\tau_jA^2_{j+1}}{2a_j^2}\left\| \frac{A_jy_j+a_jx_{j+1}}{A_{j+1}}- \frac{A_jy_j+a_jx_{j}}{A_{j+1}} \right\| ^2\\
	&\overset{\eqref{def:ak-sfista1}}{\geq} A_{j+1} \min_{x}\left[ \gamma_j\left( x \right) + \frac{\tau_jA_{j+1}}{2a_j^2} \left\| x-\tx_j\right\| ^2\right] \\
		&\overset{\eqref{tauproperty}}{=} A_{j+1}\min_{x}\left\lbrace \gamma_j(x) + \frac{L_{j+1}-\mu}{2}\|x-\tx_j\|^2\right\rbrace \\
		&\overset{\eqref{eq:min-sfista}}{=} A_{j+1}\left[ \gamma_j(y_{j+1}) + \frac{L_{j+1}-\mu}{2}\|y_{j+1}-\tx_j\|^2\right]\\
		&\overset{\eqref{def:gamma-sfista}}{=} A_{j+1}\left[\ell_{\psi_s}(y_{j+1};\tilde x_j) +\psi_n(y_{j+1})+
	\frac{L_{j+1}}{2}\|y_{j+1}-\tx_j\|^2 \right]  \\
		&\overset{\eqref{ineq check}}{\geq} A_{j+1}\left[ \psi(y_{j+1})  +
		\frac{\chi L_{j+1}}{2}\|y_{j+1}-\tx_j\|^2 \right].
	\end{align*}
	The conclusion of the lemma now follows by combining the above two relations.
	\end{proof}

\begin{lemma} \label{lm:gammakphi}
For
every $j \ge 0$, we have $\gamma_j \leq \psi$.
\end{lemma}

\begin{proof}
Define: 
    \begin{equation}\tilde \gamma_j(x):= \ell_{\psi_s}(x;\tilde x_j)  + \psi_n(x)+\frac{\mu}{2}\|x-\tilde x_j\|^2 \label{def:tgamma-sfista}.
    \end{equation}
It follows immediately from the fact that $\psi_s$ is $\mu$-convex that $\tilde \gamma_j \leq \psi$. Furthermore, immediately from the definition of $y_{j+1}$ in \eqref{eq:ynext-sfista1}, we can write:
    \begin{equation}
		y_{j+1}= \argmin_{x}\left\{\tilde{\gamma}_{j}(x)+\frac{L_{j+1}-\mu}{2}\left\|x-\tilde{x}_{j}\right\|^{2}\right\} \label{eq:min-sfista'}.
		\end{equation}
Now, clearly from \eqref{eq:min-sfista'} and the definition of $s_{j+1}$ in \eqref{eq:sk}, we see that $s_{j+1} \in \partial \tilde \gamma_j(y_{j+1})$. Furthermore, since $\tilde \gamma_j$ is $\mu$-convex, it follows from the subgradient rule for the sum of convex functions that the above inclusion is equivalent to $s_{j+1} \in \partial\left(\tilde\gamma_j(\cdot)-\frac{\mu}{2}\|\cdot-y_{j+1}\|^2\right)(y_{j+1}).$
Hence, the subgradient inequality and the fact that $\tilde \gamma_j(x) \leq \psi(x)$ imply that for all $x\in \Re^{n}$:
    \begin{align*}
    \psi(x)\geq \tilde \gamma_j(x)&\geq\tilde \gamma_j(y_{j+1})+\inner{s_{j+1}}{x-y_{j+1}}+\frac{\mu}{2}\|x-y_{j+1}\|^2=\gamma_j(x)
    \end{align*}
and thus the statement of the lemma follows.     
\end{proof}

\begin{lemma} \label{lm:hysub-4-fista}
For every $j \ge 0$ and $x \in \dom \psi_n$, we have
\[
\eta_j(x) - \eta_{j+1} (x) \ge   \frac{\chi A_{j+1}L_{j+1}}{2} \|y_{j+1} - \tx_j \|^2
\]
where
\[
\eta_j(x) := A_j [ \psi(y_j) - \psi(x) ] + \frac{\tau_j}{2} \|x-x_j\|^2.
\]
\end{lemma}

\begin{proof}
Subtracting  $A_{j+1} \psi(x)$ from both sides of the inequality in \eqref{ineq:recur} and using Lemma~\ref{lm:gammakphi} we have
\begin{align*}
A_j \psi (y_j) + & a_j\psi (x)-A_{j+1}\psi(x) + \frac{\tau_j}2 \|x_j - x\|^2 - \frac{\tau_{j+1}}2 \|x_{j+1} - x\|^2 \\
	& \ge  A_{j+1} \psi(y_{j+1})-A_{j+1}\psi(x) + \frac{\chi A_{j+1}L_{j+1}}{2} \| y_{j+1} - \tx_j \|^2 .
\end{align*}
The result now follows from the first equality in \eqref{eq:taunext-sfista1} and the definition of $\eta_j(x)$.
\end{proof}

We now state a result that will be important for deriving complexity bounds for ADAP-FISTA.
\begin{lemma} \label{lm:hysub-5-ac}
For every $j \ge 0$ and $x \in \dom \psi_n$, we have
\begin{align}\label{key convergence inequality}
 A_j [ \psi(y_j) - \psi(x) ] + \frac{\tau_j}2 \|x-x_j\|^2 \le
 \frac12 \|x-x_0\|^2 - \frac{\chi}{2} \sum_{i=0}^{j-1} A_{i+1}L_{i+1} \| y_{i+1} - \tx_i \|^2.
\end{align}
\end{lemma}
\begin{proof}
Summing the inequality of Lemma \ref{lm:hysub-4-fista} from $j=0$ to $j=j-1$, using the facts that $A_0=0$ and $\tau_0=1$, and using the definition of $\eta_j(\cdot)$ in Lemma \ref{lm:hysub-4-fista} gives us the inequality of the lemma.
\end{proof}

We are now ready to give the proof of Proposition~\ref{strongly convex fista}.
\begin{proof}[of Proposition~\ref{strongly convex fista}]
Since $\psi_s$ is $\mu$-convex, Lemma~\ref{lm:hysub-5-ac} holds. Thus, using \eqref{key convergence inequality} with $x=y_j$, it follows that for all $j\geq 0$:
\begin{equation}\label{bound diff-3}
\|y_j-x_0\|^2 \overset{\eqref{key convergence inequality}}{\geq}\chi \sum_{i=1}^{j}A_{i}L_{i} \|y_{i}-\tilde x_{i-1}\|^2 \geq \chi A_jL_j\|y_{j}-\tilde x_{j-1}\|^2.
\end{equation}
Hence, for all $j\geq 0$, relation \eqref{ineq: ineq 1} in step~4 of ADAP-FISTA is always satisfied and thus ADAP-FISTA never fails. In view of this observation and Proposition~\ref{prop:nest_complex1}, it follows that if $\psi_s$ is $\mu$-convex then ADAP-FISTA always terminates successfully with a $(y,u)$ satisfying relations \eqref{u inclusion 1} and \eqref{u termination 1} in a finite number of iterations. The inclusion $u\in (\psi_s+\psi_n)(y)$ then follows immediately from the inclusion in \eqref{u inclusion 1} and the subgradient rule for the sum of convex functions.
\end{proof}

\section{Technical Results for Proof of Lagrange Multipliers}\label{technical lagrange multiplier}
The  following basic result is used in Lemma~\ref{lem:qbounds-2}. Its proof can be found, for instance, in \cite[Lemma~A.4]{MaxJeffRen-admm}. Recall that $\nu^+_A$ denotes  the smallest positive singular value of a nonzero linear operator $A$ . 

\begin{lemma}\label{lem:linalg} 
Let $A:\Re^n \to \Re^l$ be a  nonzero linear operator. Then,
\[
\nu^+_A\|u\|\leq \|A^*u\|,   \quad \forall u \in A(\Re^n).
\]
\end{lemma}

The following technical result, whose proof can be found in Lemma 3.10 of \cite{RenWilmelo2020iteration}, plays an important role in the proof of Lemma~\ref{lem:qbounds-2} below.

\begin{lemma}\label{lem:bound_xiN}
Let $h$ be a function as in (A1).
Then, for every $\delta \ge 0$, $z\in {\mathcal H}$,   and $\xi \in \partial_{\delta} h(z)$, we have
\begin{equation}\label{bound xi}
\|\xi\|{\rm dist}(u,\partial {\mathcal H}) \le \left[{\rm dist}(u,\partial {\mathcal H})+\|z-u\|\right]M_h + \inner{\xi}{z-u}+\delta \quad \forall u \in {\mathcal H}
\end{equation}
where $\partial {\cal H}$ denotes the boundary of ${\cal H}$.
\end{lemma}

\begin{lemma}\label{lem:qbounds-2}
Assume that $h$ is a function as in condition (A1) and $A:\Re^n \to \Re^l$ is a linear operator satisfying condition (A2).
Assume also that
the triple $(z,q,r) \in \Re^{n} \times A(\Re^{n}) \times \Re^{n}$ satisfy
$r \in \partial h(z)+A^{*}q$.
Then:
\begin{itemize}
\item[(a)] there holds
\begin{equation}\label{ineq:aux9001-2}
 \bar{d}\nu_A^{+}\|q\|
\leq  2 D_h\left(M_h + \|r\| \right)  - \inner{q}{Az-b};
\end{equation}
\item[(b)] if, in addition, 
\begin{equation}\label{q form}
q=q^-+\chi(Az-b)   
\end{equation}
for some $q^-\in \Re^l$ and $\chi>0$,
then we have 
\begin{equation}\label{q bound-2}
 \|q\|\leq \max\left\{\|q^-\|,\frac{2D_h(M_h+\|r\|)}{\bar d \nu^{+}_A} \right\}.
\end{equation}
\end{itemize}
\end{lemma}	
\begin{proof}
(a) The assumption on $(z,q,r)$ implies that $r-A^{*}q \in \partial h(z)$. Hence, using the Cauchy-Schwarz inequality, the definitions of $\bar d$ and $\bar z$ in \eqref{phi* and lambda bound and bar d} and (A2), respectively, and
Lemma~\ref{lem:bound_xiN} with $\xi=r-A^{*}q$, $u=\bar z$, and $\delta=0$, we have:
 \begin{align}\label{first inequality p-2}
   \bar d\|r-A^{*}q\|-\left[\bar d+\|z-\bar z\|\right]M_h &\overset{\eqref{bound xi}}{\leq}  \inner{r-A^{*}q}{z-\bar z}\leq \|r\| \|z-\bar z\| - \inner{q}{Az-b}.
 \end{align}
 Now, using the above inequality,
 the triangle inequality, the definition of $D_h$ in (A1), and the facts that $\bar d \leq D_h$ and $\|z-\bar z\|\leq D_h$, we conclude that:
 \begin{align}\label{second inequality p-2}
 \bar d \|A^*q\| + \inner{q}{Az-b}
 &\overset{\eqref{first inequality p-2}}{\leq} \left[\bar d+\|z-\bar z\|\right]M_h + \|r\| \left(D_h + \bar d\right) \leq 2 D_h\left(M_h + \|r\| \right).
 \end{align}
 Noting the assumption that
$q \in A(\Re^n)$,
inequality
 \eqref{ineq:aux9001-2} now follows
from the above inequality
and
Lemma~\ref{lem:linalg}.

(b) Relation \eqref{q form} implies that $\inner{q}{Az-b}=\|q\|^2/\chi-\inner{q^-}{q}/\chi$, and hence that
\begin{equation}\label{q relation-2}
    \bar d \nu^{+}_A\|q\|+\frac{\|q\|^2}{\chi}\leq 2D_h(M_h+\|r\|)+\frac{\inner{q^-}{q}}{\chi}\leq 2D_h(M_h+\|r\|)+\frac{\|q\| }{\chi}\|q^-\|,
\end{equation}
where the last inequality is due to the Cauchy-Schwarz inequality.
Now, letting
$K$ denote the right hand side of \eqref{q bound-2}
and using \eqref{q relation-2},
we conclude that
\begin{equation}\label{prelim q bound-2}
\left(\bar d \nu^+_A+\frac{\|q\|}{\chi}    \right)\|q\|\overset{\eqref{q relation-2}}{\leq} \left(\frac{2D_h(M_h+\|r\|)}{K}+\frac{\|q\| }{\chi}\right)K\leq \left(\bar d \nu^+_A+\frac{\|q\|}{\chi}    \right) K,
\end{equation}
and hence that \eqref{q bound-2} holds.
\end{proof}

\end{appendices}
\scriptsize
\typeout{}
\bibliographystyle{plain}
\bibliography{Proxacc_ref}

\def\cprime{$'$}
\begin{thebibliography}{10}

\bibitem{Aybatpenalty}
N.S. Aybat and G.~Iyengar.
\newblock A first-order smoothed penalty method for compressed sensing.
\newblock {\em SIAM J. Optim.}, 21(1):287--313, 2011.

\bibitem{AybatAugLag}
N.S. Aybat and G.~Iyengar.
\newblock A first-order augmented {{L}agrangian} method for compressed sensing.
\newblock {\em SIAM J. Optim.}, 22(2):429--459, 2012.

\bibitem{florea2018accelerated}
M.~I. Florea and S.~A. Vorobyov.
\newblock An accelerated composite gradient method for large-scale composite
  objective problems.
\newblock {\em IEEE Transactions on Signal Processing}, 67(2):444--459, 2018.

\bibitem{MaxJeffRen-admm}
M.L.N. Goncalves, J.G. Melo, and R.D.C. Monteiro.
\newblock Convergence rate bounds for a proximal admm with over-relaxation
  stepsize parameter for solving nonconvex linearly constrained problems.
\newblock {\em Pac. J. Optim.}, 15(3):379--398, 2019.

\bibitem{NIPS2014_5615}
Q.~Gu, Z.~Wang, and H.~Liu.
\newblock Sparse pca with oracle property.
\newblock In {\em Advances in Neural Information Processing Systems 27}, pages
  1529--1537. Curran Associates, Inc., 2014.

\bibitem{HongPertAugLag}
D.~Hajinezhad and M.~Hong.
\newblock Perturbed proximal primal-dual algorithm for nonconvex nonsmooth
  optimization.
\newblock {\em Math. Program.}, 176:207--245, 2019.

\bibitem{YHe1}
Y.~He and R.D.C. Monteiro.
\newblock Accelerating block-decomposition first-order methods for solving
  composite saddle-point and two-player {N}ash equilibrium problems.
\newblock {\em SIAM J. Optim.}, 25(4):2182--2211, 2015.

\bibitem{YHe2}
Y.~He and R.D.C. Monteiro.
\newblock An accelerated {HPE}-type algorithm for a class of composite
  convex-concave saddle-point problems.
\newblock {\em SIAM J. Optim.}, 26(1):29--56, 2016.

\bibitem{SZhang-Pen-admm}
B.~Jiang, T.~Lin, S.~Ma, and S.~Zhang.
\newblock Structured nonconvex and nonsmooth optimization algorithms and
  iteration complexity analysis.
\newblock {\em Comput. Optim. Appl.}, 72(3):115--157, 2019.

\bibitem{KongThesis2021}
W.~Kong.
\newblock Accelerated inexact first-order methods for solving nonconvex
  composite optimization problems.
\newblock {\em arXiv:2104.09685}, April 2021.

\bibitem{CurvatureFree}
W.~Kong.
\newblock Complexity-optimal and curvature-free first-order methods for finding
  stationary points of composite optimization problems.
\newblock {\em Available on arXiv:2205.13055}, 2022.

\bibitem{fistaReport2021}
W.~Kong, J.~G. Melo, and R.D.C. Monteiro.
\newblock {FISTA and Extensions - Review and New Insights}.
\newblock {\em Optimization Online}, 2021.

\bibitem{WJRproxmet1}
W.~Kong, J.G. Melo, and R.D.C. Monteiro.
\newblock Complexity of a quadratic penalty accelerated inexact proximal point
  method for solving linearly constrained nonconvex composite programs.
\newblock {\em SIAM J. Optim.}, 29(4):2566--2593, 2019.

\bibitem{WJRComputQPAIPP}
W.~Kong, J.G. Melo, and R.D.C. Monteiro.
\newblock An efficient adaptive accelerated inexact proximal point method for
  solving linearly constrained nonconvex composite problems.
\newblock {\em Comput. Optim. Appl.}, 76(2):305--346, 2019.

\bibitem{NL-IPAL}
W.~Kong, J.G. Melo, and R.D.C. Monteiro.
\newblock Iteration-complexity of a proximal augmented {L}agrangian method for
  solving nonconvex composite optimization problems with nonlinear convex
  constraints.
\newblock {\em Available on arXiv:2008.07080}, 2020.

\bibitem{RenWilmelo2020iteration}
W.~Kong, J.G. Melo, and R.D.C. Monteiro.
\newblock Iteration-complexity of an inner accelerated inexact proximal
  augmented {L}agrangian method based on the classical {L}agrangian function.
\newblock {\em Available on arXiv:2008.00562}, 2020.

\bibitem{AIDAL}
W.~Kong and R.D.C. Monteiro.
\newblock An accelerated inexact dampened augmented {L}agrangian method for
  linearly-constrained nonconvex composite optimization problems.
\newblock {\em Available on arXiv:2110.11151}, 2020.

\bibitem{MinMax-RenWilliam}
W.~Kong and R.D.C. Monteiro.
\newblock An accelerated inexact proximal point method for solving
  nonconvex-concave min-max problems.
\newblock {\em SIAM Journal on Optimization}, 31(4):2558--2585, 2021.

\bibitem{LanRen2013PenMet}
G.~Lan and R.D.C. Monteiro.
\newblock Iteration-complexity of first-order penalty methods for convex
  programming.
\newblock {\em Math. Program.}, 138(1):115--139, Apr 2013.

\bibitem{LanMonteiroAugLag}
G.~Lan and R.D.C. Monteiro.
\newblock Iteration-complexity of first-order augmented {L}agrangian methods
  for convex programming.
\newblock {\em Math. Program.}, 155(1):511--547, Jan 2016.

\bibitem{ImprovedShrinkingALM20}
Z.~Li, P.-Y. Chen, S.~Liu, S.~Lu, and Y.~Xu.
\newblock Rate-improved inexact augmented {L}agrangian method for constrained
  nonconvex optimization.
\newblock {\em Available on arXiv:2007.01284}, 2020.

\bibitem{HybridPenaltyAugLag19}
Z.~Li and Y.~Xu.
\newblock Augmented {L}agrangian based first-order methods for convex and
  nonconvex programs: nonergodic convergence and iteration complexity.
\newblock {\em arXiv e-prints}, pages arXiv--2003, 2020.

\bibitem{Penalty2.5Xu}
Q.~Lin, R.~Ma, and Y~Xu.
\newblock Inexact proximal-point penalty methods for constrained non-convex
  optimization.
\newblock {\em Available on arXiv:1908.11518}, 2020.

\bibitem{ShiqiaMaAugLag16}
Y.F. Liu, X.~Liu, and S.~Ma.
\newblock On the nonergodic convergence rate of an inexact augmented
  {L}agrangian framework for composite convex programming.
\newblock {\em Math. Oper. Res.}, 44(2):632--650, 2019.

\bibitem{zhaosongAugLag18}
Z.~{Lu} and Z.~{Zhou}.
\newblock {Iteration-complexity of first-order augmented {L}agrangian methods
  for convex conic programming}.
\newblock {\em Available on arXiv:1803.09941}, 2018.

\bibitem{RJWIPAAL2020}
J.G. Melo, R.D.C. Monteiro, and H.~Wang.
\newblock Iteration-complexity of an inexact proximal accelerated augmented
  {L}agrangian method for solving linearly constrained smooth nonconvex
  composite optimization problems.
\newblock {\em Available on arXiv:2006.08048}, 2020.

\bibitem{MontSvaiter_fista}
R.D.C. Monteiro, C.~Ortiz, and B.F. Svaiter.
\newblock An adaptive accelerated first-order method for convex optimization.
\newblock {\em Comput. Optim. Appl.}, 64:31--73, 2016.

\bibitem{IterComplConicprog}
I.~Necoara, A.~Patrascu, and F.~Glineur.
\newblock Complexity of first-order inexact {L}agrangian and penalty methods
  for conic convex programming.
\newblock {\em Optim. Methods Softw.}, pages 1--31, 2017.

\bibitem{nesterov1983method}
Y.~E. Nesterov.
\newblock {\em Introductory lectures on convex optimization : a basic course}.
\newblock Kluwer Academic Publ., 2004.

\bibitem{nesterov2012gradient}
Y.E. Nesterov.
\newblock Gradient methods for minimizing composite functions.
\newblock {\em Math. Program.}, pages 1--37, 2012.

\bibitem{Patrascu2017}
A.~Patrascu, I.~Necoara, and Q.~Tran-Dinh.
\newblock Adaptive inexact fast augmented {L}agrangian methods for constrained
  convex optimization.
\newblock {\em Optim. Lett.}, 11(3):609--626, 2017.

\bibitem{inexactAugLag19}
M.~Sahin, A.~Eftekhari, A.~Alacaoglu, F.~Latorre, and V~Cevher.
\newblock An inexact augmented {L}agrangian framework for nonconvex
  optimization with nonlinear constraints.
\newblock {\em Available on arXiv:1906.11357}, 2019.

\bibitem{DualDescent}
K.~Sun and A.~Sun.
\newblock {Dual Descent ALM and ADMM}.
\newblock {\em Available on arXiv:2109.13214}, 2018.

\bibitem{YangyangAugLag17}
Y.~Xu.
\newblock Iteration complexity of inexact augmented {L}agrangian methods for
  constrained convex programming.
\newblock {\em Math. Program.}, 2019.

\bibitem{yao2017efficient}
Q.~Yao and J.T. Kwok.
\newblock Efficient learning with a family of nonconvex regularizers by
  redistributing nonconvexity.
\newblock {\em J. Mach. Learn. Res.}, 18:179--1, 2017.

\bibitem{YinMoreau}
J.~Zeng, W.~Yin, and D.~Zhou.
\newblock Moreau {E}nvelope {A}ugmented {L}agrangian method for {N}onconvex
  {O}ptimization with {L}inear {C}onstraints.
\newblock {\em J. Scientific Comp.}, 91(61), April 2022.

\bibitem{ErrorBoundJzhang-ZQLuo2020}
J.~Zhang and Z.-Q. Luo.
\newblock A global dual error bound and its application to the analysis of
  linearly constrained nonconvex optimization.
\newblock {\em Available on arXiv:2006.16440}, 2020.

\bibitem{ADMMJzhang-ZQLuo2020}
J.~Zhang and Z.-Q. Luo.
\newblock A proximal alternating direction method of multiplier for linearly
  constrained nonconvex minimization.
\newblock {\em SIAM J. Optim.}, 30(3):2272--2302, 2020.

\end{thebibliography}

\end{document}